\documentclass[12pt]{amsart}
\usepackage[colorlinks=true,pagebackref,hyperindex,citecolor=Green,linkcolor=Blue]{hyperref}
\usepackage[top=1in, bottom=1in, left=1.1in, right=1.1in]{geometry}
\usepackage{amsmath}
\usepackage{amsfonts}
\usepackage{amssymb}
\usepackage{amsthm}
\usepackage[all]{xy}
\usepackage{mathrsfs}
\usepackage{enumitem} 
\usepackage[T1]{fontenc}
\usepackage{color}
\usepackage{stmaryrd}
\usepackage{bigints}

\usepackage{pgf}
\usepackage{graphicx} 
\usepackage{tikz-cd} 
\usepackage{tikz} 

\makeatletter
\def\namedlabel#1#2{\begingroup
#2%
\def\@currentlabel{#2}%
\phantomsection\label{#1}\endgroup
}
\makeatother

\newtheorem{theorem}{Theorem}[section]

\newtheorem{corollary}[theorem]{Corollary}

\newtheorem{lemma}[theorem]{Lemma}

\newtheorem{proposition}[theorem]{Proposition}

\newtheorem{theoremx}{Theorem}
\newtheorem{algorithm}[theorem]{Algorithm}

\newtheorem{corollaryx}{Corollary}

\theoremstyle{definition} 
\newtheorem{definition}[theorem]{Definition}

\newtheorem{example}[theorem]{Example}
\newtheorem{remark}[theorem]{Remark}
\numberwithin{equation}{subsection}

\newcommand{\crd}{\color {red}}

\definecolor{blue-violet}{rgb}{0.54, 0.17, 0.89}
\definecolor{Blue}{rgb}{0.01, 0.28, 1.0}
\definecolor{gGreen}{rgb}{0.2, 0.8, 0.2}
\definecolor{Green}{rgb}{0.04, 0.85, 0.32}

\makeatletter
\def\@tocline#1#2#3#4#5#6#7{\relax
  \ifnum #1>\c@tocdepth 
  \else
    \par \addpenalty\@secpenalty\addvspace{#2}%
    \begingroup \hyphenpenalty\@M
    \@ifempty{#4}{%
      \@tempdima\csname r@tocindent\number#1\endcsname\relax
    }{%
      \@tempdima#4\relax
    }%
    \parindent\z@ \leftskip#3\relax \advance\leftskip\@tempdima\relax
    \rightskip\@pnumwidth plus4em \parfillskip-\@pnumwidth
    #5\leavevmode\hskip-\@tempdima
      \ifcase #1
       \or\or \hskip 1.9em \or \hskip 2em \else \hskip 3em \fi%
      #6\nobreak\relax
    \dotfill\hbox to\@pnumwidth{\@tocpagenum{#7}}\par
    \nobreak
    \endgroup
  \fi}
\makeatother

\newcommand{\ZZ}{\mathbb{Z}}
\newcommand{\QQ}{\mathbb{Q}}

\newcommand{\KK}{\mathbb{K}}

\newcommand{\cA}{\mathscr{A}}

\newcommand{\cP}{{\mathcal P}}

\newcommand{\bZ}{{\mathbb{Z}}}
\newcommand{\lra}{{\longrightarrow}}

\begin{document}

\title[Morse resolutions and Betti splittings]{Morse resolutions of monomial ideals and Betti splittings}

\author[J. \`Alvarez Montaner]{Josep \`Alvarez Montaner}
\address{Departament de Matem\`atiques  and  Institut de Matem\`atiques de la UPC-BarcelonaTech (IMTech)\\  Universitat Polit\`ecnica de Catalunya \\ Av.~Diagonal 647, Barcelona
08028; and Centre de Recerca Matem\`atica (CRM), Bellaterra, Barcelona 08193.} 
\email{Josep.Alvarez@upc.edu}

\author[M.L. Aparicio Garc\'ia]{Mar\'ia Luc\'ia Aparicio Garc\'ia}
\address{Departament de Matem\`atiques  \\  Universitat Polit\`ecnica de Catalunya \\ Av.~Diagonal 647, Barcelona
08028} 
\email{Maria.Lucia.Aparicio@estudiantat.upc.edu}

\author[A.Mafi]{Amir Mafi}
\address{Department of Mathematics, University Of Kurdistan, P.O. Box: 416, Sanandaj, Iran.}
\email{A\_Mafi@ipm.ir}

\thanks{JAM is partially supported by the projects PID2019-103849GB-I00 and PID2023-146936NB-I00 financed by the Spanish State Agency MCIN/AEI/10.13039/501100011033/ FEDER, UE, by the GEOMVAP 2021-SGR-00603 AGAUR project
and  Spanish State Research Agency, through the Severo Ochoa and Mar\'ia de Maeztu Program for Centers and Units of Excellence in R$\&$D (project CEX2020-001084-M). AM was supported by CIMPA-ICTP "Research in Pairs" programme. MLAG is supported by an INIREC grant associated to the GEOMVAP project 2021 SGR 00603}



\begin{abstract}
We use discrete Morse theory to study free resolutions of monomial ideals in combination with splitting techniques. We establish the minimality of such pruned resolutions for several classes of ideals, including stable and linear quotient ideals. In particular, we unify classical constructions such as the Eliahou-Kervaire and Herzog-Takayama resolutions within the pruned resolution framework. Additionally, we introduce methods to reduce the minimality study of a pruned resolution for an ideal to that of a smaller subideal and present a variant of our pruned resolution for powers of monomial ideals.
\end{abstract}

\maketitle

\section{Introduction}

The study of minimal free resolutions of monomial ideals in polynomial rings with coefficients in a field has garnered significant attention from commutative algebraists over the past few decades. These resolutions provide valuable information about the structure of the ideal, including its generators, relations, and higher-order syzygies.

\vskip 2mm

Explicit minimal free resolutions are often limited to specific classes of monomial ideals. For example, the {\it Eliahou-Kervaire resolution} \cite{EK} applies to {\it stable} ideals, while the {\it Herzog-Takayama resolution} \cite{HerzogTakayama} addresses ideals with {\it linear quotients}. Alternatively, non-minimal free resolutions, such as the Taylor resolution \cite{Taylor} and the Lyubeznik resolution \cite{Lyu88}, are easier to construct and fit within the broader framework of simplicial resolutions. This theory, developed by Bayer, Peeva, and Sturmfels \cite{BayerPeevaSturmfels} and later generalized to regular cellular and CW-resolutions \cite{BayerSturmfels, JollenWelker}, connects the algebraic structure of resolutions to the topology of associated complexes. However, as shown by Velasco \cite{Velasco}, not all minimal free resolutions of monomial ideals can be supported on CW-complexes.

\vskip 2mm

Using the discrete Morse theory developed by Forman \cite{Forman} and Chari \cite{Char}, Batzies and Welker proposed in \cite{BatziesWelker} a method to simplify regular cellular resolutions.  However, discrete Morse theory has a notable limitation: it cannot be applied iteratively. To overcome this, {\it algebraic discrete Morse theory}, developed independently by Sk\"oldberg \cite{Skol} and J\"ollenbeck and Welker \cite{JollenWelker}, offers a more flexible alternative.

\vskip 2mm

The key idea behind using discrete Morse theory to study free resolutions is to encode the information from the Taylor resolution (or another resolution supported on a simplicial complex or a regular cell complex)  into a directed graph, where certain edges reflect the obstruction to minimality. By systematically removing these edges in an efficient order, we generate a smaller resolution.  We loosely refer to any resolution  obtained in this way as a {\it Morse resolution}.
Examples of such resolutions may be found in \cite{BatziesWelker}, \cite{JollenWelker}, \cite{Joll}, \cite{AMFG20}, \cite{BarileMacchia},\cite{ChauKara}, \cite{EN-arXiv}, \cite{Faridipd1}, \cite{FaridiArtinian}, \cite{Edu15}, \cite{Edu}.

\vskip 2mm

In this paper, we present several extensions of the {\it pruned resolutions} introduced by the first author in collaboration with Fern\'andez-Ramos and Gimenez \cite{AMFG20}.. Section \ref{Sec2} provides a detailed overview of the original pruned resolution. In Section \ref{Sec2_3}, we propose a variation designed for powers of squarefree monomial ideals. This approach replaces the simplicial complex associated with the Taylor resolution with a simplicial complex introduced by Cooper, El Khoury, Faridi, Mayes-Tang, Morey, \c{S}ega, and Spiroff \cite{FaridiSimplicial}. The main contribution of the paper is presented in Section \ref{Sec2_2}, where we combine pruned resolutions with {\it Betti splittings}, a concept introduced by Francisco, H`a, and Van Tuyl \cite{FTVT2} (see also \cite{EK}) that provides a recursive method to compute Betti numbers.

 \vskip 2mm

   If no confusion arises with the version chosen, we will loosely use the term monomial ideal {\it admitting  a minimal pruned resolution} to indicate that there exists an order of the generators of $I$ such that the cellular resolutions obtained in either Theorem \ref{C1}, Theorem \ref{C3} or the recursive version of Theorem \ref{C2} is minimal.

 \vskip 2mm

Using the original version of the pruned resolution presented in Section \ref{Sec2} we prove that monomial ideals with at most $5$ generators admit a minimal pruned resolution (Theorem \ref{thm_q_5}). Additionally, we establish a persistence property for the minimality of subideals (see Section \ref{Minimal_11}), a result that proves useful in the development of Section \ref{Sec2_2}.
In Section \ref{Sec2_3} we present a version of a pruned resolution for powers of monomial ideals. Initial computations with small examples indicate promising potential for this approach in identifying minimal Morse resolutions for such ideals.

\vskip 2mm

In Section \ref{Sec2_2} we introduce a Morse resolution that is based on the notion of  Betti splitting. As a noteworthy side result, we give a mild generalization, for the case of vertex partitions, of a result of Bolognini \cite{Bolognini} concerning Betti splittings of componentwise linear ideals.
 In this section we also check the minimality of the pruned resolution for  several  classes of monomial ideal such as {\it lexsegment, strongly stable, stable, vertex splittable, linear quotient} or {\it componentwise linear}  (see Definition \ref{Def_ideals} for details).
They have been extensively studied in the literature and there is a hierarchy  among them  given by the following diagram:

\vskip 4mm

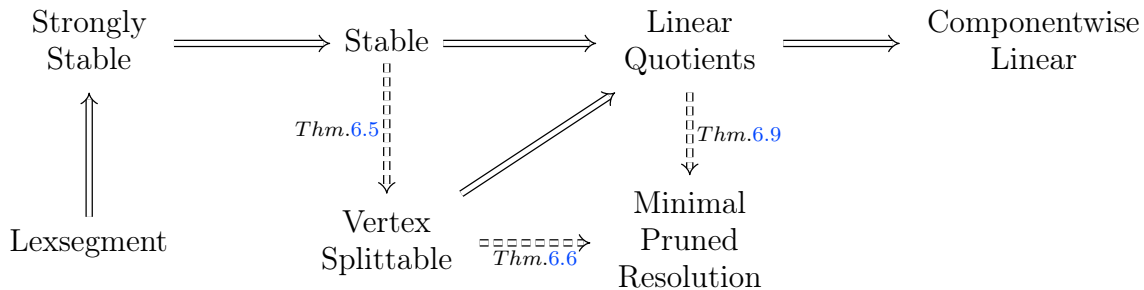
\begin{tikzcd}[row sep=large, column sep=large, arrows=Rightarrow]
\begin{tabular}{c}
Strongly \\
Stable
\end{tabular} \arrow[r] &
  \text{Stable} \arrow[r] \arrow[d,dashed, "Thm. \ref{thm1}" ']  &
 \begin{tabular}{c}
Linear \\
Quotients
\end{tabular} \arrow[r] \arrow[d,dashed, "Thm. \ref{thm3}"] & \begin{tabular}{c}
Componentwise \\
Linear
\end{tabular} \\
\text{Lexsegment} \arrow[u]   & \begin{tabular}{c}
Vertex \\
Splittable
\end{tabular} \arrow[ru] \arrow[r,dashed, "Thm. \ref{thm2}" ']  & \begin{tabular}{c}
Minimal \\ Pruned \\
Resolution
\end{tabular}
\end{tikzcd}

Our contribution is reflected in the dashed arrows. First we prove the following result. 

\newpage

\begin{theoremx}(Theorem \ref{thm1})
    \begin{itemize}
     \item Stable ideals are vertex splittable.
     \end{itemize}
\end{theoremx}

Then we prove the minimality of the pruned resolution for the following classes of ideals.

\begin{theoremx}
We have:
 \begin{itemize}
     \item Vertex splittable ideals admit a minimal pruned resolution. (Theorem \ref{thm2}) 
     \item  Linear quotient ideals admit a minimal pruned resolution. (Theorem \ref{thm3})
     \item A class of $p$-Borel fixed ideals considered by Batzies and Welker \cite{BatziesWelker} admits a minimal pruned resolution. (Corollary \ref{pBorel})
 \end{itemize}  
\end{theoremx}

This last case is a consequence of a more general result we prove in Theorem \ref{thm_product} (see also Theorem \ref{thm_product_2}). In Proposition \ref{prop_GHP} we also consider the case of ${\bf a}$-stable ideals introduced by Gasharov, Hibi and Peeva \cite{GHPeeva}.

\vskip 2mm

The upshot of our method is that it allows us to unify several classical constructions of free resolutions within the same framework of pruned resolutions. 

\begin{corollaryx}
There is a minimal pruned resolution isomorphic to:
\begin{itemize}
    \item The Eliahou-Kervaire resolution \cite{EK}. (Corollary \ref{CorEK})
    \item The Herzog-Takayama resolution \cite{HerzogTakayama}. (Corollary \ref{CorHT})
\end{itemize}
\end{corollaryx}

Moreover, there are cases where we can still apply the same techniques to determine the minimality of the pruned resolution of a monomial ideal $I$ by reducing it to the minimality of a smaller ideal $J$.

\begin{theoremx} (Theorem \ref{thm_a})
Assume
$I=J+(x_1^{a_1},\dots, x_n^{a_n})$ with $a_i \in \mathbb{Z}_{> 0}$. 
Then, $I$ admits a minimal pruned resolution if and only if $J$ admits a minimal pruned resolution.
    
\end{theoremx}

\begin{theoremx}  (Theorem \ref{thm_b})
Let $I(G)$ be the edge ideal of a graph $G$ and $x_n$ a vertex with no edges between its neighbors.  $I(G)$ admits a minimal pruned resolution if $I(G\setminus \{x_n\})$ and some intermediate edge ideals  of some hypergraphs $I(H_{k_1,\dots, k_m})$  admit a minimal pruned resolution.    
\end{theoremx}

As a consequence we prove that paths, trees, forests, cycles, wheels, and complete bipartite graphs admit minimal pruned resolution (Corollaries \ref{Ctree}, \ref{Ccycle}, \ref{Cwheel}, \ref{Cbipartite}). 

\section{Free resolutions of monomial ideals}

Let $R=\KK[x_1, \ldots, x_n]$ be the polynomial ring in $n$ variables with coefficients
in a field $\KK$. Let $I\subseteq R$ be a monomial ideal, that is an ideal generated
by monomials ${\bf x}^\alpha:=x_1^{\alpha_1}\cdots x_n^{\alpha_n}$, where $\alpha \in \ZZ_{\geq 0}^n$.
Throughout this work, we denote $|\alpha|= {\alpha_1}+\cdots +{\alpha_n}$ and $\varepsilon_1,\dots,\varepsilon_n$
will be the natural basis of $\ZZ^n$. Given a minimal set of generators $G(I):=\{m_1,\dots,m_q\}$ of $I$,
we will consider  the monomials $m_{\sigma}:={\rm lcm}(m_i \hskip 2mm | \hskip 2mm \sigma_i=1)$ for any $\sigma \in \{0,1\}^q$. Sometimes we will denote these monomials  $m_{i_1,\dots, i_s }:={\rm lcm}(m_{i_1}, \dots, m_{i_s})$ with $1\leq i_1 < \cdots < i_s \leq q$ to emphasize those $m_i$'s showing up in the expression. In this case we will also denote $\sigma=\{i_1,\dots, i_s\}$.
\vskip 2mm

A $\ZZ^n$-graded free resolution of $R/I$ is an exact sequence of free $\ZZ^n$-graded
modules:
\begin{equation}\label{resolution of I}
\mathbb{F}_{\bullet}: \hskip 3mm \xymatrix{ 0 \ar[r]& F_{p}
\ar[r]^{d_{p}}& \cdots \ar[r]& F_1 \ar[r]^{d_1}& F_{0} \ar[r]& R/I
\ar[r]& 0},
\end{equation}
where the $i$-th term is of the form
$$F_i =\bigoplus_{\alpha \in \ZZ^n} R(-\alpha)^{\beta_{i,\alpha}}\,.$$
We say that $\mathbb{F}_{\bullet}$ is minimal if the matrices of the
homogeneous morphisms $d_i: F_i\longrightarrow F_{i-1}$ do not contain
invertible elements. In this case, the exponents $\beta_{i,\alpha}$ form a
set of invariants of $R/I$ known as its {\it multigraded Betti numbers}.
Throughout this work, we will mainly consider the coarser $\ZZ$-graded free resolution. In this case, we
will encode the $\ZZ$-graded Betti numbers in the  {\it Betti diagram} of $R/I$ where the entry
on the $i$th row and $j$th column of the table is $ \beta_{i,i+j}$, that is:

\begin{center}
 $\begin{matrix}
&0&1&2&\cdots \\ \text{total:}& \text{.} &\text{.}&\text{.}& \\
\text{0:} & \beta_{0,0} & \beta_{1,1} & \beta_{2,2} & \cdots \\\text{1:}& \beta_{0,1} & \beta_{1,2} & \beta_{2,3} & \cdots \\
\vdots& \vdots & \vdots & \vdots \\\end{matrix}
$
\end{center}
We have $ \beta_{0,0}=1$ and $ \beta_{i,i}=0$ for all $i>0$. Moreover, each row of the Betti table encodes the Betti numbers of each {\it linear strand} of $\mathbb{F}_{\bullet}$. We say that $I$ has a {\it linear resolution} if $\mathbb{F}_{\bullet}$ only has one linear strand (not counting the first row).

\subsection{Cellular resolutions}

A CW-complex $X$ is a topological space obtained by attaching cells of increasing dimensions
to a discrete 
set of points $X^{(0)}$.  Let $X^{(i)}$ denote the set of $i$-cells of  $X$ and
consider the set of all cells $X^{(\ast)}:=\bigcup_{i\geq 0} X^{(i)}$. Then, we can view $X^{(\ast)}$
as a poset with the partial order  given by $\sigma' \leq \sigma$ if and only if $\sigma'$
is contained in the closure of $\sigma$. We can also give a $\ZZ^n$-graded structure $X$ by means of an
order-preserving map $gr: X^{(\ast)} \longrightarrow \ZZ_{\geq 0}^n$.


\vskip 2mm

We say that the free resolution (\ref{resolution of I}) is {\it cellular} (or is a {\it CW-resolution})
if there exists a $\ZZ^n$-graded CW-complex $(X, gr)$ such that, for all $i\geq 1$:

\begin{itemize}
 \item[$\cdot$]
 there exists a basis $\{e_\sigma\}$ of $F_i$ indexed by the $(i-1)$-cells of $X$, such that if
 $e_\sigma \in R(-\alpha)^{\beta_{i,\alpha}}$ then  $gr(\sigma)=\alpha$, and

 \item[$\cdot$]
The differential $d_i: F_i\longrightarrow F_{i-1}$ is given
by
 $$e_\sigma \hskip 2mm \mapsto \sum_{\sigma \geq \sigma' \in X^{(i-1)}} [\sigma: \sigma'] \hskip 1mm {\bf x}^{gr(\sigma) - gr(\sigma') } \hskip 1mm e_{\sigma'}\ ,\quad \forall \sigma \in X^{(i)}$$
where $[\sigma: \sigma']$ denotes the coefficient of $\sigma'$ in the image of $\sigma$ by the differential map in the cellular homology
of $X$. We have  $[\sigma: \sigma']=\pm 1$ when $(X, gr)$ is a {\it regular} CW-complex.
\end{itemize}
In the sequel, whenever we want to emphasize such a cellular structure, we will denote the free resolution as
$\mathbb{F}_{\bullet}=\mathbb{F}_{\bullet}^{(X,gr)}$. If $X$ is a simplicial complex, we say
that the free resolution is {\it simplicial}. 

\vskip 2mm
We mention here some examples of cellular resolutions that will be relevant in our work. We start with some classical simplicial resolutions which are not minimal in general. 

\vskip 2mm

$\cdot$ To define the {\it Taylor resolution  \cite{Taylor}}  $\mathbb{F}_{\bullet}^{(X_{\tt Taylor}, gr)}$ we consider the full simplicial complex on $r$ vertices, $X_{\tt Taylor}$, whose faces are labelled by
$\sigma \in \{0,1\}^r$ or, equivalently, by the corresponding monomials $m_\sigma$. We have a natural
$\ZZ^n$-grading on $X_{\tt Taylor}$ by assigning $gr(\sigma)=\alpha \in \ZZ^n$ where ${\bf x}^\alpha=m_\sigma$.

\vskip 2mm

$\cdot$ The {\it Lyubeznik resolution \cite{Lyu88}} $\mathbb{F}_{\bullet}^{(X_{\tt Lyub}, gr)}$ is a refinement of the Taylor resolution supported 
on a simplicial subcomplex $X_{\tt Lyub} \subseteq X_{\tt Taylor}$. 
We point out that it depends on the order of the generators of the ideal $I$.

\vskip 2mm
The following classical free resolutions are minimal for certain classes of monomial ideals. These resolutions were constructed prior to the introduction of cellular resolutions, but it was later shown that they belong to this class.

\vskip 2mm

$\cdot$  Eliahou and Kervaire   \cite{EK}  introduced the notion of {\it stable ideals} (see Definition \ref{Def_ideals}) and constructed an explicit minimal free resolution for this class of monomial ideals.
Mermin \cite{Mermin} demonstrated that it admits a regular cellular structure (see also  \cite{Sinefako}, \cite{OkazakiYanagawa}).

\vskip 2mm

Charambalous and Evans \cite{ChaEv} observed that the {\it Eliahou-Kervaire resolution} can be constructed as iterated mapping cones. This approach involves considering the short exact sequences arising from adding one generator of $I$
 at a time and iteratively constructing the resolution as a mapping cone of an appropriate map between complexes that were previously constructed.

\vskip 2mm

$\cdot$  Herzog and Takayama  \cite{HerzogTakayama} built upon this mapping cone procedure for constructing a free resolution of a monomial ideal. They proved that it provides a minimal free resolution for {\it linear quotient ideals} with respect to some ordering of its generators (see Definition \ref{Def_ideals}). They even describe explicitly the differentials of the resolution under some technical extra assumption \cite[Theorem 1.12]{HerzogTakayama}. 
Dochtermann and Mohammadi \cite{DochFat} and independently Goodarzi \cite{Goodarzi} proved that the {\it Herzog-Takayama resolution} admits a regular cellular structure.

\subsection{Discrete Morse theory}

Batzies and Welker \cite{BatziesWelker} (see also \cite{Welker_survey}) studied cellular resolutions using {\it discrete Morse theory} introduced by
Forman  \cite{Forman} and reformulated in terms of acyclic matchings by Chari in \cite{Char}. We briefly recall this approach but encourage the interested reader to take a look at the original sources for more insight.

\vskip 2mm

Let  $G_X$ be the directed graph were the vertices are the cells of a regular $\bZ^n$-graded CW-complex $(X,gr)$ and the edges are given by
$$E_X=\{ \sigma \lra \sigma' \hskip 2mm | \hskip 2mm  \sigma' \leq \sigma, \hskip 2mm   \dim \sigma' = \dim \sigma -1 \}.$$
For a given set of edges $\cA \subseteq E_X$, denote by $G_X^{\cA}$ the graph obtained by reversing the direction
of the edges in $\cA$, i.e., the directed graph with edges\footnote{For the sake of clarity, the arrows that we reverse will be denoted
by $\Longrightarrow$.}
$$E_X^{\cA}= (E_X \setminus \cA) \cup \{ \sigma' \Longrightarrow \sigma \hskip 2mm | \hskip 2mm  \sigma \lra \sigma' \in \cA \}.$$
%
%
When each cell of $X$ occurs in at most one edge of $\cA$, we say that $\cA$ is a {\it matching} on $X$.
A matching $\cA$ is {\it acyclic} if the associated graph $G_X^{\cA}$ is acyclic, i.e., does not
contain any directed cycle.
Given an acyclic matching  $\cA$ on $X$, the $\cA$-{\it critical cells} of $X$ are the cells of $X$ that are
not contained in any edge of $\cA$.
Finally, an acyclic matching $\cA$ is {\it homogeneous} whenever $gr(\sigma)=gr(\sigma')$
for any edge $\sigma \lra \sigma' \in \cA$.
%
%

\begin{proposition}[{\cite[Proposition 1.2]{BatziesWelker}}]
Let $(X,gr)$ be a regular $\bZ^n$-graded CW-complex and $\cA$ a homogeneous acyclic matching. Then, there is a
{\rm(}not necessarily regular{\rm)} CW-complex $X_{\cA}$
whose $i$-cells are in one-to-one correspondence with the $\cA$-critical $i$-cells of $X$, such that $X_{\cA}$ is
homotopically equivalent to $X$, and that inherits the $\bZ^n$-graded structure.
\end{proposition}

In the theory of cellular resolutions, we have:

\begin{theorem}[{\cite[Theorem 1.3]{BatziesWelker}}] \label{Morse_resolution}
Let $I\subseteq R=\KK[x_1,\dots, x_n]$ be a monomial ideal. Assume that $(X,gr)$ is a regular $\bZ^n$-graded
 CW-complex that defines a cellular resolution $\mathbb{F}_{\bullet}^{(X,gr)}$ of $R/I$. Then,  for
a homogeneous acyclic  matching $\cA$ on $G_X$, the $\bZ^n$-graded CW-complex $(X_{\cA},gr)$ supports a cellular
resolution $\mathbb{F}_{\bullet}^{(X_{\cA},gr)}$ of $R/I$. 
\end{theorem}

\vskip 2mm

We will refer to a resolution obtained in this way as a {\it Morse resolution}. The differentials of the Morse resolution
$\mathbb{F}_{\bullet}^{(X_{\cA},gr)}$ can be explicitly described
in terms of the  differentials of $\mathbb{F}_{\bullet}^{(X,gr)}$
(see \cite[Lemma 7.7]{BatziesWelker}). Namely, given an edge $\sigma \longrightarrow \sigma'$ in $E_X$ we set
$$m(\sigma, \sigma')=\begin{cases}
-\frac{1}{[\sigma : \sigma']}, & \textrm{if} \; \sigma \longrightarrow \sigma' \in \cA,\\
[\sigma : \sigma'], & \textrm{otherwise}.
\end{cases}$$
For a path $\cP : \sigma_1 \longrightarrow \sigma_2 \longrightarrow \cdots \longrightarrow \sigma_{t-1} \longrightarrow \sigma_t$ in $G_X^{\cA}$  we set 
$$m(\cP)= m(\sigma_1, \sigma_2) \cdots m(\sigma_{t-1}, \sigma_t).$$ Then, the differential $d_i^{\cA}: F_i^{(X_{\cA},gr)}\longrightarrow F_{i-1}^{(X_{\cA},gr)}$ is given,  $\forall \sigma \in X_{\cA}^{(i)}$, 
by:

\begin{equation} \label{differential_Morse}
e_\sigma \hskip 2mm \mapsto \sum_{\sigma \geq \sigma' \in X_{\cA}^{(i-1)}} [\sigma: \sigma'] \hskip 1mm \sum_{\sigma'' \in X_{\cA}^{(i-1)}} \hskip 1mm  \sum_{\cP: \sigma' \longrightarrow \sigma'' \\ \textrm{path}} \hskip 1mm m(\cP) \hskip 1mm  {\bf x}^{gr(\sigma) - gr(\sigma'') } \hskip 1mm e_{\sigma''}\ .
\end{equation}

\vskip 2mm

We point out  that the CW-complex $(X_{\cA},gr)$ that we obtain for a given homogeneous acyclic matching $\cA$
is not necessarily regular. Therefore, we cannot always iterate the procedure.
 Sk\"oldberg \cite{Skol} and J\"ollenbeck and Welker \cite{JollenWelker} introduced  {\it algebraic discrete Morse theory}
to overcome this issue. The main difference is that they work directly with an initial free resolution $\mathbb{F}_{\bullet}$
without paying attention whether it has or not a cellular structure.

 \vskip 2mm

Given a basis $X=\bigcup_{i\geq 0} X^{(i)}$ of the corresponding free modules $F_i$,
we consider the directed graph $G_X$ on the set of basis elements with the corresponding
set of edges $E_X$. Then,  we may define an acyclic matching $\cA \subseteq E_X$ (see \cite[Definition 2.1]{JollenWelker})
but, in this case, we have to make sure that the coefficient $[\sigma : \sigma']$ in the differential
corresponding to an edge $\sigma \lra \sigma' \in \cA$ is a unit. We consider the $\cA$-critical
basis elements $X_\cA$ and construct a free resolution $\mathbb{F}_{\bullet}^{X_\cA}$ that is
homotopically equivalent to $\mathbb{F}_{\bullet}$ (see \cite[Theorem 2.2]{JollenWelker}).

\vskip 2mm

Most of the Morse resolutions found in the literature start with the Taylor simplicial complex $(X_{\tt Taylor}, gr)$. We collect some of them:

\vskip 2mm
$\cdot$ Batzies and Welker \cite{BatziesWelker} obtained a minimal Morse resolutions for generic monomial ideals as well as stable and linear quotient monomial ideals. They also gave a minimal Morse resolution for powers of the homogeneous maximal ideal.

\vskip 2mm

$\cdot$ J\"ollenbeck \cite{Joll} constructed a Morse resolution of the edge ideal of a graph indexed by the {\it non-broken circuits} of the graph. He also constructed a Morse resolution for monomial ideals satisfying the {\it strong gcd-condition}. 


\vskip 2mm

$\cdot$ J\"ollenbeck and Welker \cite{JollenWelker}  described a minimal Morse resolution for principal stable (Borel fixed) monomial ideals. They also give a minimal Morse resolution for a class of $p$-Borel that includes Cohen-Macaulay $p$-Borel fixed ideals. 

\vskip 2mm

$\cdot$ The first author of this paper in collaboration with Fern\'andez-Ramos and Gimenez \cite{AMFG20} developed 
a very simple method to produce acyclic matchings from the Taylor resolution. This method, named {\it pruned resolution}, is explained in Section \ref{Sec2} but, in general, it does not produce minimal resolutions.

\vskip 2mm

$\cdot$ Barile and Macchia \cite{BarileMacchia} produced a minimal Morse resolution for edge ideals of forests. 
Their method has been generalized to any monomial ideal by Chau and Kara \cite{ChauKara}. However the Morse resolutions obtained in this way are not minimal in general.

\vskip 2mm

$\cdot$
Cooper, El Khoury, Faridi, Mayes-Tang, Morey, \c{S}ega and Spiroff \cite{Faridipd1} gave a Morse resolution of powers of ideals with projective dimension at most one. In an unpublished paper, Engstr\"om and Nor\'en \cite{EN-arXiv} gave a Morse resolution of powers of edge ideals of paths.



\section{Pruned resolution of monomial ideals} \label{Sec2}

In this section we will review the pruned resolution introduced in \cite{AMFG20}. 
Let $R=\KK[x_1, \ldots, x_n]$ be a polynomial ring  over a field $\KK$.
Let $I=\langle m_1,\dots, m_q \rangle \subseteq R$ be a monomial ideal and consider its Taylor resolution $\mathbb{F}_{\bullet}^{(X_{\tt Taylor}, gr)}$. The \textit{pruned resolution} introduced in \cite{AMFG20} is a cellular free resolution $\mathbb{F}_{\bullet}^{(X_{\cA_P},gr)}$ of $R/I$ supported on a  $\mathbb{Z}^n$-graded CW-complex $(X_{\cA_P},gr)$, where  $\cA_P \subseteq E_{X_{\tt Taylor}}$ is  the set of pruned edges obtained using the following:

\vskip 2mm

\begin{algorithm}\label{alg1} 

\vskip 2mm

{\rm \noindent {\sc Input:} The set of edges $E_{X_{\tt Taylor}}$.


\vskip 2mm



For $j$ from $1$ to $r$, incrementing by $1$

\vskip 2mm

\begin{itemize}

\item[\textbf{(j)}] ${\it Prune}$ the edge ${\sigma} \lra{\sigma + \varepsilon_j}$ for
all $\sigma \in \{0,1\}^q$ such that $\sigma_j=0$, where `prune'
means remove the edge\footnote{When we remove an edge, we also remove its two vertices and all the edges
passing through these two vertices.}
if it survived after step $(j-1)$ and
$gr(\sigma)=gr(\sigma + \varepsilon_j) $.

\end{itemize}

\vskip 2mm

\noindent {\sc Return:} The set $\cA_P$ of edges that have been pruned.

}

\end{algorithm}

It is proved in \cite{AMFG20} that $\cA_P$ is a homogeneous acyclic matching, $X_{\tt Taylor}$ and thus, 

\begin{theorem}\label{C1}
Let $I \subseteq R=\KK[x_1, \ldots, x_n]$ be a monomial ideal and
$\cA_P \subseteq E_{X_{\tt Taylor}}$ be the set of pruned edges obtained using Algorithm \ref{alg1}.
Then, the $\bZ^n$-graded CW-complex $(X_{\cA_P},gr)$
supports a cellular free resolution $\mathbb{F}_{\bullet}^{(X_{\cA_P},gr)}$ of $R/I$.
\end{theorem}

The $\ZZ^n$-graded modules appearing in the free resolution are of the form $\bigoplus_{\alpha \in \ZZ^n} R(-\alpha)^{\overline{\beta}_{i,\alpha}}\,$ and we refer to 
${\overline{\beta}_{i,\alpha}}$ as the \textit{pruned Betti numbers} of $I$.

\vskip 2mm

Some remarks are in order:

\begin{remark} The pruned resolution presents the following features:
\begin{itemize}
\item[$i)$] In general it is not minimal. We have ${{\beta}_{i,\alpha}} \leq {\overline{\beta}_{i,\alpha}}$ for all $i$ and $\alpha$.

\item[$ii)$] It depends on the order of the generators of the ideal $I$.

\item[$iii)$] It is not sensitive to the characteristic of the base field. We may find examples where the pruned resolution is not minimal in characteristic zero but it is in some characteristic different from zero.

\item[$iv)$] We can perform partial pruning by considering appropriate subsets of pruned edges $\cA\subseteq \cA_P$. This was already considered in \cite{AMFG20} to obtain the Lyubeznik resolution or a pruned simplicial resolution.

\item[$v)$] Instead of starting with the Taylor resolution or, equivalently, the simplicial complex ${(X_{\tt Taylor}, gr)}$, we may apply the pruning algorithm to a smaller simplicial complex or even a regular CW-complex. This approach will be used in Section .

\end{itemize}

\end{remark}

\begin{remark} Let $I=(m_1,\dots,m_q)$ be a monomial ideal with $q\geq 3$. The construction of the Taylor graph $E_{X_{\text{Taylor}}}$ can be done  iteratively:


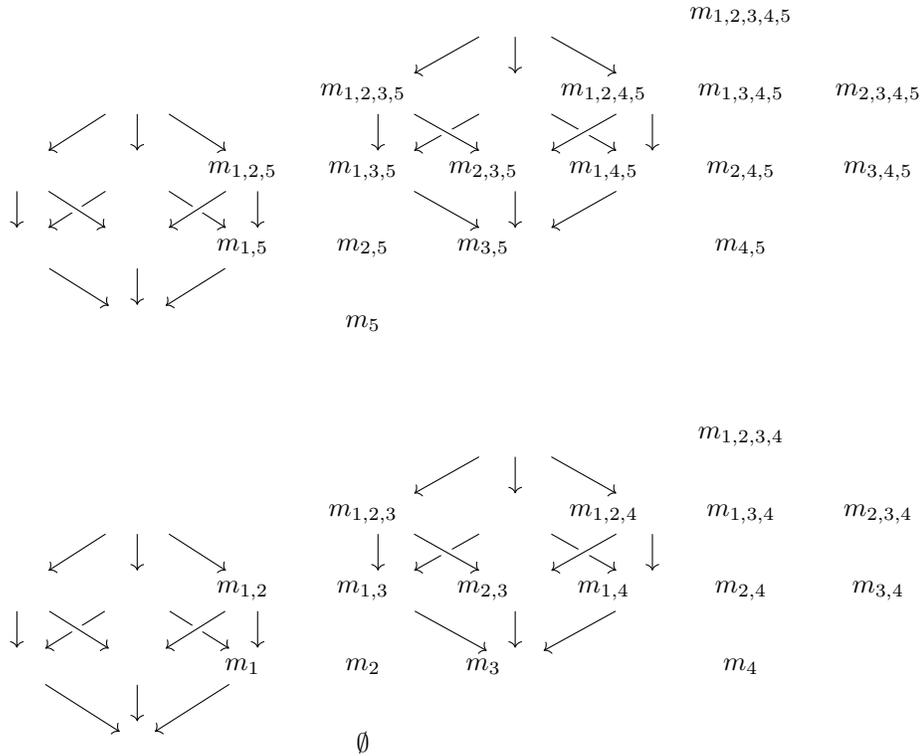
\begin{tikzpicture}[on top/.style={preaction={draw=white,-,line width=#1}},
    on top/.default=4pt]
 \hskip 3cm  \matrix (m) [
    matrix of nodes,
    column sep=3mm,
    row sep=5mm,
   font=\footnotesize,
  ] {
                 &               &             &             &                 &                 &             &               &             &               &  &   \\
                 &               &             &             & ${ m_{1,2,3,4,5}}$ &                 &             &  &             &  &  &  \\
                 & ${ m_{1,2,3,5}}$ &          & ${ m_{1,2,4,5}}$ & ${ m_{1,3,4,5}}$   & ${ m_{2,3,4,5}}$     &  &   & &  &  &  \\
     ${ m_{1,2,5}}$ & $m_{1,3,5}$     & $m_{2,3,5}$ & ${ m_{1,4,5}}$   & $m_{2,4,5}$     & ${ m_{3,4,5}}$       &   &  &  &               &    &   \\
     $m_{1,5}$   &  $m_{2,5}$    & $m_{3,5}$   &             & $m_{4,5}$       &                 &             &      &             &               &               &   \\
                 & $m_5$         &             &             &                 &                 &             &               &             &               &               &   \\
                 &               &             &             &                 &                 &             &               &             &               &  &   \\
                 &               &             &             & ${ m_{1,2,3,4}}$   &                 &             &  &             &  & &  \\
                 & $m_{1,2,3}$   &             & $m_{1,2,4}$ & $m_{1,3,4}$     & $m_{2,3,4}$     &  &   & &  &   &  \\
     $m_{1,2}$   & $m_{1,3}$     & $m_{2,3}$   & $m_{1,4}$   & $m_{2,4}$       & $m_{3,4}$       &    &     &   &               &     &   \\
     $m_{1}$     &  $m_2$        & $m_3$       &             & $m_4$           &                 &             &        &             &               &               &   \\
                 & $\emptyset$          &             &             &                 &                 &             &               &             &               &               &   \\
  };
  \begin{scope}[
    font=\small,
    inner sep=.25em,
    every node/.style={fill=white},
  ]
  \draw [->] (m-11-1) -- (m-12-2);
  \draw [->] (m-11-3) -- (m-12-2);
  \draw [->] (m-11-2) -- (m-12-2);
  \draw [->] (m-10-2) -- (m-11-1);
  \draw [->] (m-10-1) -- (m-11-1);
  \draw [->] (m-10-2) -- (m-11-3);
  \draw [->] (m-10-3) -- (m-11-3);
  \draw [->] (m-9-2) -- (m-10-2);
  \draw [->] (m-9-2) -- (m-10-1);
  \draw [->] (m-9-2) -- (m-10-3);
  \draw [->, on top] (m-10-1) -- (m-11-2);
  \draw [->, on top] (m-10-3) -- (m-11-2);
  \draw [->] (m-10-4) -- (m-11-5);
  \draw [->] (m-10-6) -- (m-11-5);
  \draw [->] (m-10-5) -- (m-11-5);
  \draw [->] (m-9-5) -- (m-10-4);
  \draw [->] (m-9-4) -- (m-10-4);
  \draw [->] (m-9-5) -- (m-10-6);
  \draw [->] (m-9-6) -- (m-10-6);
  \draw [->] (m-8-5) -- (m-9-5);
  \draw [->] (m-8-5) -- (m-9-4);
  \draw [->] (m-8-5) -- (m-9-6);
  \draw [->, on top] (m-9-4) -- (m-10-5);
  \draw [->, on top] (m-9-6) -- (m-10-5);


    \draw [->] (m-5-1) -- (m-6-2);
    \draw [->] (m-5-3) -- (m-6-2);
    \draw [->] (m-5-2) -- (m-6-2);
    \draw [->] (m-4-2) -- (m-5-1);
    \draw [->] (m-4-1) -- (m-5-1);
    \draw [->] (m-4-2) -- (m-5-3);
    \draw [->] (m-4-3) -- (m-5-3);
    \draw [->] (m-3-2) -- (m-4-2);
    \draw [->] (m-3-2) -- (m-4-1);
    \draw [->] (m-3-2) -- (m-4-3);
    \draw [->, on top] (m-4-1) -- (m-5-2);
    \draw [->, on top] (m-4-3) -- (m-5-2);
    \draw [->] (m-4-4) -- (m-5-5);
    \draw [->] (m-4-6) -- (m-5-5);
    \draw [->] (m-4-5) -- (m-5-5);
    \draw [->] (m-3-5) -- (m-4-4);
    \draw [->] (m-3-4) -- (m-4-4);
    \draw [->] (m-3-5) -- (m-4-6);
    \draw [->] (m-3-6) -- (m-4-6);
    \draw [->] (m-2-5) -- (m-3-5);
    \draw [->] (m-2-5) -- (m-3-4);
    \draw [->] (m-2-5) -- (m-3-6);
    \draw [->, on top] (m-3-4) -- (m-4-5);
    \draw [->, on top] (m-3-6) -- (m-4-5);
 
  
  \end{scope}
\end{tikzpicture}


$\cdot$ Begin with the \textit{bottom-left cube}, which represents the Taylor graph for  $(m_1, m_2, m_3)$.

\vskip 2mm
$\cdot$
    To incorporate $m_4$, duplicate the existing graph (from the base case) by adding a \textit{bottom-right cube}. 
    The edges in the fourth direction (connecting nodes involving $m_4$) are not drawn in the diagram to avoid clutter. These omitted edges would connect 
    $ \emptyset \longleftarrow m_4, \quad m_1 \longleftarrow m_{1,4}, \quad m_{1,2} \longleftarrow m_{1,2,4}, \quad \dots \quad, \quad m_{1,2,3} \longleftarrow m_{1,2,3,4}.
  $
\vskip 2mm
$\cdot$
     To incorporate $m_5$, duplicate the graph again to form two \textit{upper cubes} (left and right). These represent the Taylor graph for $(m_1, m_2, m_3, m_4, m_5)$. 
    Similar to the fourth direction, edges in the fifth direction are also omitted to simplify the visualization.

\vskip 2mm

The Taylor graph for any monomial ideal $ (m_1, \dots, m_q)$ can be constructed by repeating this process. The key idea is to iteratively duplicate the existing structure and extend in a new direction corresponding to the next monomial, while optionally omitting edges in higher dimensions for clarity. 

\vskip 2mm

The first three directions of the pruning Algorithm \ref{alg1} are $( \nwarrow )$, $( \uparrow )$, and $( \nearrow )$, respectively. Following the iterative method to represent the graph, we prune edges in the corresponding cubes first to the \textit{right}, then \textit{upstairs}, and so on. This approach has a significant advantage from a computational perspective. Specifically, if we prune an edge $ \sigma \longrightarrow \sigma + \varepsilon_j $, we also prune all the edges of the form:
    \[
    \sigma + \varepsilon_k \longrightarrow \sigma + \varepsilon_j + \varepsilon_k, \quad \text{where } k > j.
    \]
Consequently, during the process of duplicating cubes, we only need to duplicate the edges (and their corresponding vertices) that have not been pruned in previous steps.


\end{remark}


  


\subsection{Minimal pruned resolution} \label{Minimal_1}
Let $\mathbb{F}_{\bullet}^{(X_{\cA},gr)}$ be the cellular resolution that we obtain in Theorem \ref{C1}. Since it depends on the order of the generators of the monomial ideal we will consider the following notion of minimality.

\begin{definition}
Let $I \subseteq R$ be a monomial ideal. We say that it {\it admits a minimal pruned resolution} if there exists an order of the generators of $I$ such that the pruned free resolution that we obtain in Theorem \ref{C1} is minimal
\end{definition}

The pruned free resolution 
$\mathbb{F}_{\bullet}^{(X_{\cA},gr)}$  is minimal if and only if for all $\sigma'\in X^{(i-1)}_{\cA}$ and $\sigma \in X^{(i)}_{\cA}$ we have $gr(\sigma') \neq gr(\sigma)$
whenever $\sigma' \leq \sigma$ or there is no path or a path for which the differential $d_i^\cA(e_\sigma)$, described in Equation \ref{differential_Morse}, has non zero coefficient for the basis element $e_{\sigma'}$ whenever $gr(\sigma') = gr(\sigma)$. This was already mentioned by Batzies and Welker \cite[Lemma 7.5]{BatziesWelker} for any Morse resolution. 

\vskip 2mm

The main  result in this section shows that an ideal with a few number of generators admits a minimal pruned resolution.

\begin{theorem} \label{thm_q_5}
Let $I=(m_1,\dots,m_q)\subseteq R$ be a squarefree monomial ideal. If $q\leq 5$ then $I$ admits a minimal pruned resolution.    
\end{theorem}

\begin{proof}
We have to check that either there are no cells $\sigma''\in X^{(i-1)}_{\cA}$ and $\sigma \in X^{(i)}_{\cA}$ such that $m_{\sigma''}= m_{\sigma}$ but $\sigma \not\geq \sigma''$ or, if this is the case, there is no path $\sigma'' \longrightarrow \sigma'$  in $G_{X}^{\cA}$ with $\sigma\geq \sigma'\in X^{(i-1)_{\cA}}$ and thus the coefficient of differential $d_i^{\cA}(e_\sigma)$ is zero for the basis element $e_{\sigma''}$. 

\vskip 2mm
The result trivially holds true for $q\leq 3$. For $q=4$ we may consider, without loss of generality, that $\sigma''=(1,1,0,0)=\{1,2\}\in X^{(2)}_{\cA}$ and $\sigma=(1,0,1,1)=\{1,3,4\}\in X^{(3)}_{\cA}$. Assume that $m_{1,2}=m_{1,3,4}$ and thus $m_2 | m_{1,3,4}$, $m_3 | m_{1,2}$ and $m_4 | m_{1,2}$. Therefore 
$$m_{1,2}=m_{1,2,3}=m_{1,2,4}=m_{1,3,4}=m_{1,2,3,4}$$ and we get a contradiction since the pruning algorithm would force to prune both $\sigma''=\{1,2\}$ and $\sigma=\{1,3,4\}$ unless the case that we have previously pruned $\{1,2,3\}$ with either $\{2,3\}$ or $\{1,3\}$, $\{1,2,4\}$ with either $\{2,4\}$ or $\{1,4\}$ and $\{1,2,3,4\}$ with $\{2,3,4\}$. 
Then, the cells $\sigma''$ and $\sigma$ survive to the pruning algorithm but there is no path in $G_X^{\cA}$ from $\sigma'\leq \sigma$ being either $\{1,3\}, \{1,4\}$ or $\{3,4\}$ to $\sigma''$ and thus the pruned resolution is minimal. 
We point out that this last case is achieved, for example, for the ideal $I=(x_1 x_2 x_3 x_4,x_1 x_2 x_4 x_5 x_6,x_1 x_3 x_5, x_2 x_3 x_6)$.

\vskip 2mm

For $q=5$ we have to consider the following cases: 

\vskip 2mm

$\bullet$ Consider cells in $X^{(3)}_{\cA}$ and $X^{(4)}_{\cA}$. Without loss of generality, assume  $\sigma''=(1,1,1,0,0)=\{1,2,3\}\in X^{(3)}_{\cA}$ and $\sigma=(1,1,0,1,1)=\{1,2,4,5\}\in X^{(4)}_{\cA}$. In this case we have$$m_{1,2,3}=m_{1,2,3,4}=m_{1,2,3,5}=m_{1,2,4,5}=m_{1,2,3,4,5}$$ and we proceed with a similar argument as above. The pruning algorithm would force to prune both $\sigma''=\{1,2,3\}$ and $\sigma=\{1,2,4,5\}$ unless the case that we have previously pruned $\{1,2,3,4\}$ with either $\{2,3,4\}$, $\{1,3,4\}$ or $\{1,2,4\}$, $\{1,2,3,5\}$ with either $\{2,3,5\}$, $\{1,3,5\}$ or $\{1,2,5\}$ and $\{1,2,3,4,5\}$ with either $\{2,3,4,5\}$ or $\{1,3,4,5\}$. The cells $\sigma''$ and $\sigma$ survive to the pruning algorithm but there is no path in $G_X^{\cA}$ from $\sigma'\leq \sigma$ being either $\{1,2,4\}, \{1,2,5\}, \{1,4,5\}$ or $\{2,4,5\}$ to $\sigma''$ and thus the pruned resolution is minimal. 

\vskip 2mm

$\bullet$ Consider cells in $X^{(2)}_{\cA}$ and $X^{(3)}_{\cA}$.  There are two possible configurations that, without loss of generality, we assume to be: 

\begin{itemize}
    \item[$\cdot$] $\sigma''=(1,1,0,0,0)=\{1,2\}\in X^{(2)}_{\cA}$ and $\sigma=(1,0,1,1,0)=\{1,3,4\}\in X^{(3)}_{\cA}$. Then we have:
$$m_{1,2}=m_{1,2,3}=m_{1,2,4}=m_{1,3,4}=m_{1,2,3,4}$$
This case is analogous to the case $q=4$ and the pruned resolution is minimal. 
\vskip 2mm
    \item[$\cdot$] $\sigma''=(1,1,0,0,0)=\{1,2\}\in X^{(2)}_{\cA}$ and $\sigma=(0,0,1,1,1)=\{3,4,5\}\in X^{(3)}_{\cA}$. Then we have: 
    \begin{align*}
     m_{1,2} &=m_{1,2,3}= m_{1,2,4}= m_{1,2,5} = m_{1,2,3,4}=m_{1,2,3,5} =m_{1,2,4,5}  \\ &= m_{3,4,5}=m_{1,3,4,5}=m_{2,3,4,5}= m_{1,2,3,4,5}.  \end{align*}
     and we get a contradiction because the pruning algorithm forces to prune $\{3,4,5\}$ and $\{1,3,4,5\}$ at the first step. 
\end{itemize}
\end{proof}

For $q=6$ generators we may find examples where the pruned resolution is not minimal. For instance, we present an example of a triangulation of the projective plane, whose minimal free resolution depends on the characteristic of the base field (see \cite{KimuraTeraiYoshida}). As a result, we cannot obtain a minimal pruned resolution in characteristic zero.

\begin{example} \label{ex_q_6}.
 Let $I\subseteq \KK[x_1,\dots ,x_{10}]$ be the squarefree monomial ideal
 $$I=(x_1x_2x_8x_9x_{10}, x_2x_3x_4x_5x_{10},  x_5x_6x_7x_8x_{10}, x_1x_4x_5x_6x_9, x_1x_2x_3x_6x_7, x_3x_4x_7x_8x_9).$$ The pruning steps of Algorithm \ref{alg1} are indicated in red in the following diagram:

 \begin{tikzpicture}[on top/.style={preaction={draw=white,-,line width=#1}},
    on top/.default=4pt]
 \hskip -1cm  \matrix (m) [
    matrix of nodes,
    column sep=1mm,
    row sep=2mm,
   font=\tiny,
  ] {
                 &               &             &             &                 &                 &             &               &             &               & ${\crd m_{1,2,3,4,5,6}}$ &   \\
                 &               &             &             & ${\crd m_{1,2,3,4,5}}$ &                 &             & ${\crd m_{1,2,3,5,6}}$ &             & ${\crd m_{1,2,4,5,6}}$ & ${\crd m_{1,3,4,5,6}}$ & ${\crd m_{2,3,4,5,6}}$ \\
                 & ${\crd m_{1,2,3,5}}$ &          & ${\crd m_{1,2,4,5}}$ & ${\crd m_{1,3,4,5}}$   & ${\crd m_{2,3,4,5}}$     & ${\crd m_{1,2,5,6}}$ & ${\crd m_{1,3,5,6}}$   & ${\crd m_{2,3,5,6}}$ & \fbox{$m_{1,4,5,6}$}   & ${\crd m_{2,4,5,6}}$   & ${\crd m_{3,4,5,6}}$ \\
     ${\crd m_{1,2,5}}$ & $m_{1,3,5}$     & $m_{2,3,5}$ & ${\crd m_{1,4,5}}$   & $m_{2,4,5}$     & ${\crd m_{3,4,5}}$       & $m_{1,5,6}$   & ${\crd m_{2,5,6}}$     & ${\crd m_{3,5,6}}$   &               & $m_{4,5,6}$     &   \\
     $m_{1,5}$   &  $m_{2,5}$    & $m_{3,5}$   &             & $m_{4,5}$       &                 &             & $m_{5,6}$         &             &               &               &   \\
                 & $m_5$         &             &             &                 &                 &             &               &             &               &               &   \\
                 &               &             &             &                 &                 &             &               &             &               & ${\crd m_{1,2,3,4,6}}$ &   \\
                 &               &             &             & ${\crd m_{1,2,3,4}}$   &                 &             & ${\crd m_{1,2,3,6}}$ &             & ${\crd m_{1,2,4,6}}$ & ${\crd m_{1,3,4,6}}$ & ${\crd m_{2,3,4,6}}$ \\
                 & \fbox{$m_{1,2,3}$}   &             & $m_{1,2,4}$ & $m_{1,3,4}$     & $m_{2,3,4}$     & $m_{1,2,6}$ & ${\crd m_{1,3,6}}$   & $m_{2,3,6}$ & ${\crd m_{1,4,6}}$   & ${\crd m_{2,4,6}}$   & $m_{3,4,6}$ \\
     $m_{1,2}$   & $m_{1,3}$     & $m_{2,3}$   & $m_{1,4}$   & $m_{2,4}$       & $m_{3,4}$       & $m_{1,6}$   & $m_{2,6}$     & $m_{3,6}$   &               & $m_{4,6}$     &   \\
     $m_{1}$     &  $m_2$        & $m_3$       &             & $m_4$           &                 &             & $m_6$         &             &               &               &   \\
                 & $\emptyset$          &             &             &                 &                 &             &               &             &               &               &   \\
  };
  \begin{scope}[
    font=\footnotesize,
    inner sep=.25em,
    every node/.style={fill=white},
  ]
  \draw (m-11-1) -- (m-12-2);
  \draw (m-11-3) -- (m-12-2);
  \draw (m-11-2) -- (m-12-2);
  \draw (m-10-2) -- (m-11-1);
  \draw (m-10-1) -- (m-11-1);
  \draw (m-10-2) -- (m-11-3);
  \draw (m-10-3) -- (m-11-3);
  \draw (m-9-2) -- (m-10-2);
  \draw (m-9-2) -- (m-10-1);
  \draw (m-9-2) -- (m-10-3);
  \draw[on top] (m-10-1) -- (m-11-2);
  \draw[on top] (m-10-3) -- (m-11-2);
  \draw (m-10-4) -- (m-11-5);
  \draw (m-10-6) -- (m-11-5);
  \draw (m-10-5) -- (m-11-5);
  \draw (m-9-5) -- (m-10-4);
  \draw (m-9-4) -- (m-10-4);
  \draw (m-9-5) -- (m-10-6);
  \draw (m-9-6) -- (m-10-6);
  \draw (m-8-5) -- (m-9-5);
  \draw (m-8-5) -- (m-9-4);
  \draw[draw=red, <-] (m-8-5) -- (m-9-6);
  \draw[on top] (m-9-4) -- (m-10-5);
  \draw[on top] (m-9-6) -- (m-10-5);
  \draw (m-10-7) -- (m-11-8);
  \draw (m-10-9) -- (m-11-8);
  \draw (m-10-8) -- (m-11-8);
  \draw (m-9-8) -- (m-10-7);
  \draw (m-9-7) -- (m-10-7);
  \draw (m-9-8) -- (m-10-9);
  \draw (m-9-9) -- (m-10-9);
  \draw[draw=red, <-] (m-8-8) -- (m-9-8);
  \draw (m-8-8) -- (m-9-7);
  \draw (m-8-8) -- (m-9-9);
  \draw[on top] (m-9-7) -- (m-10-8);
  \draw[on top] (m-9-9) -- (m-10-8);
  \draw (m-9-10) -- (m-10-11);
  \draw (m-9-12) -- (m-10-11);
  \draw (m-9-11) -- (m-10-11);
  \draw[draw=red, <-] (m-8-11) -- (m-9-10);
  \draw (m-8-10) -- (m-9-10);
  \draw (m-8-11) -- (m-9-12);
  \draw (m-8-12) -- (m-9-12);
  \draw (m-7-11) -- (m-8-11);
  \draw (m-7-11) -- (m-8-10);
  \draw[draw=red, <-] (m-7-11) -- (m-8-12);
  \draw[draw=red, <-, on top] (m-8-10) -- (m-9-11);
  \draw[on top] (m-8-12) -- (m-9-11);

    \draw (m-5-1) -- (m-6-2);
    \draw (m-5-3) -- (m-6-2);
    \draw (m-5-2) -- (m-6-2);
    \draw (m-4-2) -- (m-5-1);
    \draw (m-4-1) -- (m-5-1);
    \draw (m-4-2) -- (m-5-3);
    \draw (m-4-3) -- (m-5-3);
    \draw (m-3-2) -- (m-4-2);
    \draw (m-3-2)[draw=red, <-] -- (m-4-1);
    \draw (m-3-2) -- (m-4-3);
    \draw[on top] (m-4-1) -- (m-5-2);
    \draw[on top] (m-4-3) -- (m-5-2);
    \draw (m-4-4) -- (m-5-5);
    \draw (m-4-6) -- (m-5-5);
    \draw (m-4-5) -- (m-5-5);
    \draw (m-3-5) -- (m-4-4);
    \draw[draw=red, <-] (m-3-4) -- (m-4-4);
    \draw[draw=red, <-] (m-3-5) -- (m-4-6);
    \draw (m-3-6) -- (m-4-6);
    \draw (m-2-5) -- (m-3-5);
    \draw (m-2-5) -- (m-3-4);
    \draw[draw=red, <-] (m-2-5) -- (m-3-6);
    \draw[on top] (m-3-4) -- (m-4-5);
    \draw[on top] (m-3-6) -- (m-4-5);
    \draw (m-4-7) -- (m-5-8);
    \draw (m-4-9) -- (m-5-8);
    \draw (m-4-8) -- (m-5-8);
    \draw (m-3-8) -- (m-4-7);
    \draw (m-3-7) -- (m-4-7);
    \draw[draw=red, <-] (m-3-8) -- (m-4-9);
    \draw (m-3-9) -- (m-4-9);
    \draw (m-2-8) -- (m-3-8);
    \draw (m-2-8) -- (m-3-7);
    \draw[draw=red, <-] (m-2-8) -- (m-3-9);
    \draw[draw=red, <-, on top] (m-3-7) -- (m-4-8);
    \draw[on top] (m-3-9) -- (m-4-8);
    \draw (m-3-10) -- (m-4-11);
    \draw (m-3-12) -- (m-4-11);
    \draw (m-3-11) -- (m-4-11);
    \draw (m-2-11) -- (m-3-10);
    \draw (m-2-10) -- (m-3-10);
    \draw[draw=red, <-] (m-2-11) -- (m-3-12);
    \draw (m-2-12) -- (m-3-12);
    \draw (m-1-11) -- (m-2-11);
    \draw (m-1-11) -- (m-2-10);
    \draw[draw=red, <-] (m-1-11) -- (m-2-12);
    \draw[draw=red, <-, on top] (m-2-10) -- (m-3-11);
    \draw[on top] (m-2-12) -- (m-3-11);
  \end{scope}
\end{tikzpicture}

We point out that all the pruned edges in the graph correspond to monomials $m_\alpha= x_1\cdots x_{10}$. Among those elements  surviving to the algorithm we find  
$$m_{1,2,3}= m_{1,4,5,6}=x_1\cdots x_{10}$$ that we boxed out. Therefore, there exists $\sigma' \in X_{\cA}^{(3)}$ with $m_{1,2,3} = m_{\sigma'}$ and $\sigma \in X_{\cA}^{(4)}$ with $m_{1,4,5,6} = m_{\sigma}$ that are not comparable. However, we have a path $\cP : \sigma \longrightarrow \sigma'$ given by $$  m_{1,4,5,6} \longrightarrow m_{1,4,5} \longrightarrow m_{1,3,4,5} \longrightarrow m_{1,3,5} \longrightarrow m_{1,2,3,5}\longrightarrow m_{1,2,3}$$
such that the differential $d_i^\cA(e_\sigma)$, described in Equation \ref{differential_Morse}, has  coefficient $m(\cP)=2$ for the basis element $e_{\sigma'}$. Thus, the pruned resolution is minimal if the characteristic of the base field is two but it is not minimal otherwise.
\end{example}

\subsection{Minimal pruned resolution of subideals} \label{Minimal_11}
Admitting a minimal pruned resolution is not always preserved by subideals. However, we present some specific cases where this property holds, which will be useful in later sections.

 \begin{proposition} \label{prop_a} Let $I=(m_1,\dots,m_q)\subseteq R$ be a monomial ideal and consider the ideal $J=(m_1,\dots,m_{q-1}) \subset R$. If 
$I$ admits a minimal pruned resolution where no pruning have been performed at the $q$-th step of the pruning Algorithm \ref{alg1} then $J$ has a minimal pruned resolution.
\end{proposition}

\begin{proof}
 If there is no pruning at the last step of Algorithm \ref{alg1} and we get a minimal pruned resolution, it means that the restriction to the ideal $J$ has to provide a minimal pruned resolution as well.
\end{proof}

The previous result is no longer true without the assumption of no pruning at the $q$-th step.
\begin{example}
Let $J$ be the ideal considered in Example \ref{ex_q_6}. It does not admit a minimal pruned resolution if the characteristic of the base field is zero. However, the ideal $I=J+(x_2,x_3,x_8)$ admits a minimal pruned resolution.
\end{example}

\begin{corollary} \label{cor_a}
Let $I=(m_1,\dots,m_q)\subseteq R$ be a monomial ideal  and  consider the ideal $J=(m_1,\dots,m_{q-1}) \subset R$. If $I$ admits a minimal pruned resolution and $m_q=x_1^{a_1}\cdots x_n^{a_n}$ is a monomial with  $a_i$ larger than the degree of $x_i$ in any generator of $J$ (or $a_i=0$), then  $J$ admits a minimal pruned resolution.  
\end{corollary}

\begin{proof}
 The condition on the degrees of the monomial $m_q=x_1^{a_1}\cdots x_n^{a_n}$ ensures that we will not have a pruning  at the last step of Algorithm \ref{alg1}. Thus the result follows from Proposition \ref{prop_a}.
\end{proof}

Let $I\subseteq R$ be a monomial ideal minimally generated by $G(I)=\{m_1,\dots, m_q\}$. Given  a monomial $m\in R$, let  $I_{\leq m}:=(\{m_i \in G(I) \; | \; m_i \,{\rm divides}\, m\})$. Gasharov, Hibi and Peeva \cite[Theorem 2.1]{GHPeeva} described a minimal free resolution of  $I_{\leq m}$ from a minimal free resolution of  $I$. The following result could be considered an analogue in the context of pruned resolutions.

 \begin{corollary} \label{prop_peeva} Let $I\subseteq R$ be a monomial ideal and $m\in R$ a monomial. If 
$I$ admits a minimal pruned resolution then $I_{\leq m}$ also admits a minimal pruned resolution.
\end{corollary}

\begin{proof}
Without loss of generality assume that $I_{\leq m}=(m_1,\dots,m_r)$, $r\leq q$. Given the divisibility condition that defines this ideal we have that there is no pruning between the vertices in the Taylor graph corresponding to the ideal $I_{\leq m}$ and those involving the monomials $m_{r+1}, \dots, m_q$. Therefore, the ideal  $I_{\leq m}$ admits a minimal pruned resolution whenever $I$ does so.
\end{proof}

\section{Pruned resolution of powers of monomial ideals} \label{Sec2_3}
In this section we present a Morse resolution of powers of squarefree monomial ideals that starts, instead of the Taylor simplicial complex,  with a simplicial complex introduced by Cooper, El Khoury, Faridi, Mayes-Tang, Morey, \c{S}ega and Spiroff \cite{FaridiSimplicial}.

\vskip 2mm

Let $R=\KK[x_1, \ldots, x_n]$ be a polynomial ring  over a field $\KK$ and let $I=(m_1,\dots,m_q) \subseteq R$ be a squarefree monomial ideal. 
In this section we are interested in free resolutions of powers $I^r$, with $r\geq 1$. A set of generators of  $I^r$ is 
$$\{ m_{i_1} \cdots  m_{i_r} \; | \; 1\leq i_1 \leq \cdots \leq i_r \leq q\} $$ and thus the Taylor complex  $X_{{\tt Taylor}(I^r)}$ associated to this ideal has dimension $\binom{q+r-1}{r} -1$. Having a large number of generators makes the pruning algorithm (Algorithm \ref{alg1}) more difficult to execute, and it is more likely to produce a resolution that is far from minimal.

\vskip 2mm

In \cite[Theorem 5.9]{FaridiSimplicial} the authors defined a simplicial complex $$\mathbb{L}^r_q \subseteq X_{{\tt Taylor}(I^r)}$$
which supports a free resolution $\mathbb{F}_{\bullet}^{(\mathbb{L}^r_q,gr)}$ of any squarefree monomial ideal $I$. They even went further and used discrete Morse theory to find a smaller simplicial complex $\mathbb{L}^r(I) \subseteq \mathbb{L}^r_q $ that also supports a free resolution. In this section we will use a pruning algorithm on $\mathbb{L}^r_q$ that will lead to a cellular resolution.

\vskip 2mm

We start reviewing the construction of $\mathbb{L}^r_q$ given in \cite[Proposition 4.3]{FaridiSimplicial}. 
Recall that the vertices of the Taylor complex of $I$ are indexed by $\sigma \in \{0,1\}^q$ with the standard vectors  $\varepsilon_1,\dots, \varepsilon_q$ corresponding to the generators. We now identify the generators of $I^r$ with the elements of $$\mathcal{N}^r_q:=\{\sigma=(\sigma_1,\dots,\sigma_q) \in \ZZ_{\ge 0}^q \; | \;  |\sigma| =r\}.$$ From this set we can construct $X_{{\tt Taylor}(I^r)}$ by considering the corresponding lcm's. The subcomplex $\mathbb{L}^r_q$ is defined as follows.


\begin{definition}
The facets of the simplicial complex $\mathbb{L}^r_q \subseteq X_{{\tt Taylor}(I^r)}$  are:
$$
\mathbb{L}^r_q=\begin{cases}
\langle F^r_1, F^r_2, \dots, F^r_q, G^r_1,\dots, G^r_q\rangle &
\text{if $r>3$ and $q\ge 2$}\\
\langle B^r, G^r_1,\dots, G^r_q\rangle &
\text{if $r = 3$ and $q\geq 2$ or $r = 2$ and $q > 2$ }\\
\langle G^2_1, G^2_2\rangle &
\text{if $r = 2$ and $q=2$}\\
\langle \mathcal{N}^r_q \rangle &
\text{if $r=1$ or $q=1$.}\\
\end{cases}
$$  
where, setting $s=\lceil \frac{r}{2}\rceil$, we have:
\begin{align*}
F^r_i = & \{\sigma \in \mathcal{N}^r_q   \; | \; \sigma_i \le \max\{r-1, s\} \mbox{ and } \sigma_j \leq s \mbox{ for } j \neq i\},\\ 
G^r_i=&\{\sigma \in \mathcal{N}^r_q  \; | \; \sigma_i \ge r-1\}, \\
B^r= &\left\{\sigma \in \mathcal{N}^r_q   \; | \; \sigma_i \le s \mbox{ for all } i\in \{1,\dots,q\} \right\}. 
\end{align*}
\end{definition}

\begin{example}
Let $I\subseteq R$ be a monomial ideal with $q=3$ generators. The simplicial complex $\mathbb{L}^2_3 \subseteq X_{{\tt Taylor}(I^2)}$ is
\vskip 2mm

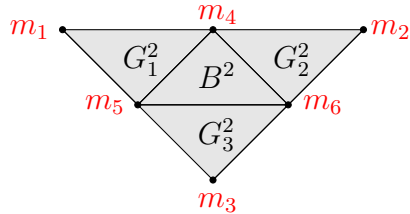
$$\begin{tikzpicture}
\tikzstyle{point}=[inner sep=0pt]
\coordinate (a) at (0,1); 
\coordinate (b) at (-1,0);
\coordinate (c) at (1,0) ;
\coordinate (d) at (-2,1);
\coordinate (e) at (2,1) ;
\coordinate (f) at (0,-1) ;
\draw [fill=gray!20](a.center) -- (b.center) -- (c.center);
\draw [fill=gray!20](a.center) -- (b.center) -- (d.center);
\draw [fill=gray!20](a.center) -- (c.center) -- (e.center);
\draw [fill=gray!20](b.center) -- (c.center) -- (f.center);
\draw (a.center) -- (b.center);
\draw (a.center) -- (c.center);
\draw (a.center) -- (d.center);
\draw (a.center) -- (e.center);
\draw (b.center) -- (c.center);
\draw (b.center) -- (d.center);
\draw (b.center) -- (f.center);
\draw (c.center) -- (e.center);
\draw (c.center) -- (f.center);
\pgfputat{\pgfxy(-2.7,1)}{\pgfbox[left,center]{${\crd m_1}$}}
\pgfputat{\pgfxy(2.1,1)}{\pgfbox[left,center]{${\crd m_2}$}}
\pgfputat{\pgfxy(-0.2,-1.3)}{\pgfbox[left,center]{${\crd m_3}$}}
\pgfputat{\pgfxy(-0.2,1.2)}{\pgfbox[left,center]{${\crd m_4}$}}
\pgfputat{\pgfxy(-1.7,0)}{\pgfbox[left,center]{${\crd m_5}$}}
\pgfputat{\pgfxy(1.2,0)}{\pgfbox[left,center]{${\crd m_6}$}}
\pgfputat{\pgfxy(-1.2,.6)}{\pgfbox[left,center]{$G_1^2$}}
\pgfputat{\pgfxy(.8,.6)}{\pgfbox[left,center]{$G_2^2$}}
\pgfputat{\pgfxy(-.2,-.4)}{\pgfbox[left,center]{$G_3^2$}}
\pgfputat{\pgfxy(-.2,.4)}{\pgfbox[left,center]{$B^2$}}
\draw[black, fill=black] (a) circle(0.04);
\draw[black, fill=black] (b) circle(0.04);
\draw[black, fill=black] (c) circle(0.04);
\draw[black, fill=black] (d) circle(0.04);
\draw[black, fill=black] (e) circle(0.04);
\draw[black, fill=black] (f) circle(0.04);
\end{tikzpicture}$$

\vskip 2mm
\noindent where we denote in red  the vertices corresponding to the generators of $I^2$. We visualize the graph of $\mathbb{L}^2_3$ as the subgraph of $ G_{X_{{\tt Taylor}(I^2)}}$ given by the boxed vertices 

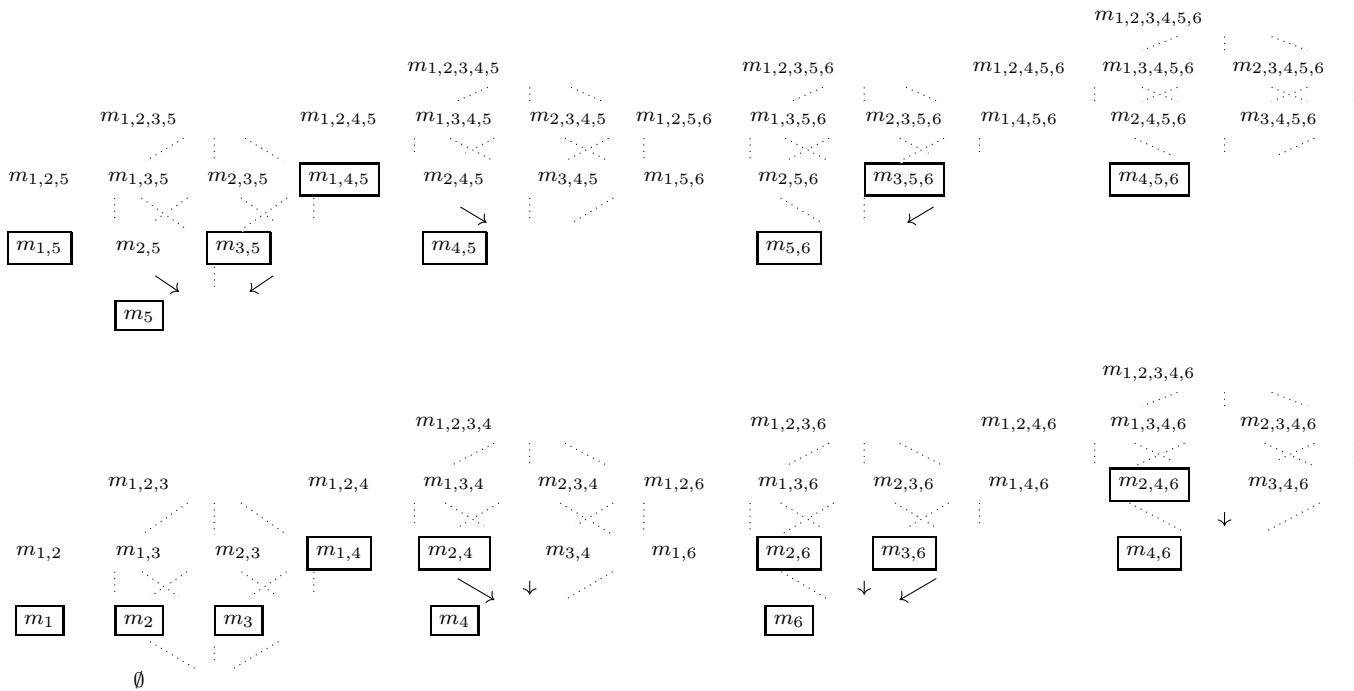
\begin{tikzpicture}[on top/.style={preaction={draw=white,-,line width=#1}},
    on top/.default=2pt]
 \hskip -1cm  \matrix (m) [
    matrix of nodes,
    column sep=1mm,
    row sep=2mm,
   font=\tiny,
  ] {
                 &               &             &             &                 &                 &             &               &             &               & ${ m_{1,2,3,4,5,6}}$ &   \\
                 &               &             &             & ${ m_{1,2,3,4,5}}$ &                 &             & ${ m_{1,2,3,5,6}}$ &             & ${ m_{1,2,4,5,6}}$ & ${ m_{1,3,4,5,6}}$ & ${ m_{2,3,4,5,6}}$ \\
                 & ${ m_{1,2,3,5}}$ &          & ${ m_{1,2,4,5}}$ & ${ m_{1,3,4,5}}$   & ${ m_{2,3,4,5}}$     & ${ m_{1,2,5,6}}$ & ${ m_{1,3,5,6}}$   & ${ m_{2,3,5,6}}$ & $m_{1,4,5,6}$   & ${ m_{2,4,5,6}}$   & ${ m_{3,4,5,6}}$ \\
     ${ m_{1,2,5}}$ & $m_{1,3,5}$     & $m_{2,3,5}$ & \fbox{${ m_{1,4,5}}$}   & $m_{2,4,5}$     & ${ m_{3,4,5}}$       & $m_{1,5,6}$   & ${ m_{2,5,6}}$     & \fbox{${ m_{3,5,6}}$}   &               & \fbox{$m_{4,5,6}$}     &   \\
     \fbox{$m_{1,5}$}   &  $m_{2,5}$    & \fbox{$m_{3,5}$}   &             & \fbox{$m_{4,5}$}       &                 &             & \fbox{$m_{5,6}$}         &             &               &               &   \\
                 & \fbox{$m_5$}         &             &             &                 &                 &             &               &             &               &               &   \\
                 &               &             &             &                 &                 &             &               &             &               & ${ m_{1,2,3,4,6}}$ &   \\
                 &               &             &             & ${ m_{1,2,3,4}}$   &                 &             & ${ m_{1,2,3,6}}$ &             & ${ m_{1,2,4,6}}$ & ${ m_{1,3,4,6}}$ & ${ m_{2,3,4,6}}$ \\
                 & $m_{1,2,3}$   &             & $m_{1,2,4}$ & $m_{1,3,4}$     & $m_{2,3,4}$     & $m_{1,2,6}$ & ${ m_{1,3,6}}$   & $m_{2,3,6}$ & ${ m_{1,4,6}}$   & \fbox{$m_{2,4,6}$}  & $m_{3,4,6}$ \\
     $m_{1,2}$   & $m_{1,3}$     & $m_{2,3}$   & \fbox{$m_{1,4}$}   & \fbox{$m_{2,4}$ }      & $m_{3,4}$       & $m_{1,6}$   & \fbox{$m_{2,6}$}     & \fbox{$m_{3,6}$}   &               & \fbox{$m_{4,6}$}     &   \\
     \fbox{$m_{1}$}     &  \fbox{$m_2$}        & \fbox{$m_3$}       &             & \fbox{$m_4$}           &                 &             & \fbox{$m_6$}         &             &               &               &   \\
                 & $\emptyset$          &             &             &                 &                 &             &               &             &               &               &   \\
  };
  \begin{scope}[
    font=\footnotesize,
    inner sep=.25em,
    every node/.style={fill=white},
  ]
  \draw  [dotted](m-11-1) -- (m-12-2);
  \draw  [dotted](m-11-3) -- (m-12-2);
  \draw [dotted](m-11-2) -- (m-12-2);
  \draw [dotted](m-10-2) -- (m-11-1);
  \draw [dotted](m-10-1) -- (m-11-1);
  \draw [dotted](m-10-2) -- (m-11-3);
  \draw [dotted](m-10-3) -- (m-11-3);
  \draw [dotted](m-9-2) -- (m-10-2);
  \draw [dotted](m-9-2) -- (m-10-1);
  \draw [dotted](m-9-2) -- (m-10-3);
  \draw[dotted,on top] (m-10-1) -- (m-11-2);
  \draw[dotted,on top] (m-10-3) -- (m-11-2);
  \draw [->](m-10-4) -- (m-11-5);
  \draw [dotted](m-10-6) -- (m-11-5);
  \draw [->] (m-10-5) -- (m-11-5);
  \draw [dotted](m-9-5) -- (m-10-4);
  \draw [dotted](m-9-4) -- (m-10-4);
  \draw [dotted](m-9-5) -- (m-10-6);
  \draw [dotted](m-9-6) -- (m-10-6);
  \draw [dotted](m-8-5) -- (m-9-5);
  \draw [dotted](m-8-5) -- (m-9-4);
  \draw [dotted] (m-8-5) -- (m-9-6);
  \draw[dotted,on top] (m-9-4) -- (m-10-5);
  \draw[dotted, on top] (m-9-6) -- (m-10-5);
  \draw [dotted](m-10-7) -- (m-11-8);
  \draw [->](m-10-9) -- (m-11-8);
  \draw [->](m-10-8) -- (m-11-8);
  \draw [dotted](m-9-8) -- (m-10-7);
  \draw [dotted](m-9-7) -- (m-10-7);
  \draw [dotted](m-9-8) -- (m-10-9);
  \draw [dotted](m-9-9) -- (m-10-9);
  \draw[dotted] (m-8-8) -- (m-9-8);
  \draw [dotted](m-8-8) -- (m-9-7);
  \draw [dotted](m-8-8) -- (m-9-9);
  \draw[dotted,on top] (m-9-7) -- (m-10-8);
  \draw[dotted,on top] (m-9-9) -- (m-10-8);
  \draw [dotted](m-9-10) -- (m-10-11);
  \draw [dotted](m-9-12) -- (m-10-11);
  \draw [->] (m-9-11) -- (m-10-11);
  \draw [dotted] (m-8-11) -- (m-9-10);
  \draw [dotted](m-8-10) -- (m-9-10);
  \draw [dotted](m-8-11) -- (m-9-12);
  \draw [dotted](m-8-12) -- (m-9-12);
  \draw[dotted] (m-7-11) -- (m-8-11);
  \draw[dotted] (m-7-11) -- (m-8-10);
  \draw[dotted] (m-7-11) -- (m-8-12);
  \draw[dotted, on top] (m-8-10) -- (m-9-11);
  \draw[dotted,on top] (m-8-12) -- (m-9-11);

    \draw [->](m-5-1) -- (m-6-2);
    \draw [->](m-5-3) -- (m-6-2);
    \draw [dotted](m-5-2) -- (m-6-2);
    \draw [dotted](m-4-2) -- (m-5-1);
    \draw [dotted](m-4-1) -- (m-5-1);
    \draw [dotted](m-4-2) -- (m-5-3);
    \draw [dotted](m-4-3) -- (m-5-3);
    \draw [dotted](m-3-2) -- (m-4-2);
    \draw [dotted](m-3-2) -- (m-4-1);
    \draw [dotted](m-3-2) -- (m-4-3);
    \draw[dotted, on top] (m-4-1) -- (m-5-2);
    \draw[dotted, on top] (m-4-3) -- (m-5-2);
    \draw [->](m-4-4) -- (m-5-5);
    \draw [dotted](m-4-6) -- (m-5-5);
    \draw [dotted](m-4-5) -- (m-5-5);
    \draw [dotted](m-3-5) -- (m-4-4);
    \draw[dotted] (m-3-4) -- (m-4-4);
    \draw[dotted] (m-3-5) -- (m-4-6);
    \draw [dotted](m-3-6) -- (m-4-6);
    \draw [dotted](m-2-5) -- (m-3-5);
    \draw [dotted](m-2-5) -- (m-3-4);
    \draw[dotted] (m-2-5) -- (m-3-6);
    \draw[dotted, on top] (m-3-4) -- (m-4-5);
    \draw[dotted, on top] (m-3-6) -- (m-4-5);
    \draw [dotted](m-4-7) -- (m-5-8);
    \draw [->](m-4-9) -- (m-5-8);
    \draw [dotted](m-4-8) -- (m-5-8);
    \draw [dotted](m-3-8) -- (m-4-7);
    \draw[dotted] (m-3-7) -- (m-4-7);
    \draw[dotted] (m-3-8) -- (m-4-9);
    \draw [dotted](m-3-9) -- (m-4-9);
    \draw [dotted](m-2-8) -- (m-3-8);
    \draw [dotted](m-2-8) -- (m-3-7);
    \draw[dotted] (m-2-8) -- (m-3-9);
    \draw[dotted, on top] (m-3-7) -- (m-4-8);
    \draw[dotted,on top] (m-3-9) -- (m-4-8);
    \draw [dotted](m-3-10) -- (m-4-11);
    \draw [dotted](m-3-12) -- (m-4-11);
    \draw [dotted](m-3-11) -- (m-4-11);
    \draw [dotted](m-2-11) -- (m-3-10);
    \draw [dotted](m-2-10) -- (m-3-10);
    \draw [dotted] (m-2-11) -- (m-3-12);
    \draw [dotted](m-2-12) -- (m-3-12);
    \draw [dotted](m-1-11) -- (m-2-11);
    \draw [dotted](m-1-11) -- (m-2-10);
    \draw[dotted] (m-1-11) -- (m-2-12);
    \draw[dotted, on top] (m-2-10) -- (m-3-11);
    \draw[dotted,on top] (m-2-12) -- (m-3-11);
  \end{scope}
\end{tikzpicture} 
\end{example}

\vskip 2mm

In general, given a monomial ideal $I\subseteq R$ we may apply the pruning algorithm to the subcomplex $\mathbb{L}^r_q \subseteq X_{{\tt Taylor}(I^r)}$ to compute a cellular resolution of $I^r$.

 \vskip 2mm
 
 \begin{algorithm}\label{alg3} 

\vskip 2mm

{\rm \noindent {\sc Input:} The set of edges $E_{\mathbb{L}^r_q}$. 


\vskip 2mm

Apply Algorithm \ref{alg1} to  $E_{\mathbb{L}^r_q}$.

\vskip 2mm

\noindent {\sc Return:} The set $\cA_L$ of edges that have been pruned.

}

\end{algorithm}

We interpret this algorithm as a partial pruning. In particular, $\cA_L$ is a homogeneous acyclic matching on $E_{\mathbb{L}^r_q}$ and thus, 
 
 \begin{theorem}\label{C3}
Let $I \subseteq R$ be a monomial ideal and
$\cA_L \subseteq E_{\mathbb{L}^r_q}$ be the set of pruned edges obtained using Algorithm \ref{alg3}.
Then, the $\bZ^n$-graded CW-complex $(X_{\cA_L},gr)$
supports a cellular free resolution $\mathbb{F}_{\bullet}^{(X_{\cA_L},gr)}$ of $R/I^r$.
\end{theorem}


\begin{definition}
Let $I \subseteq R$ be a monomial ideal. We say that the $r$-th power $I^r$ {\it admits a minimal pruned resolution} if there exists an order of the generators of $I^r$ such that the pruned free resolution that we obtain in Theorem \ref{C3} is minimal
\end{definition}

\begin{example}
 Consider the squarefree monomial ideal $I=(xy,xz,yz)\subseteq R=\KK[x,y,z]$. The Betti tables of the free resolutions $\mathbb{F}_{\bullet}^{(\mathbb{L}^r_q,gr)}$ and $\mathbb{F}_{\bullet}^{(X_{\cA_L},gr)}$ are respectively

 \vskip 2mm
\begin{verbatim}
         0  1  2  3                    0  1  2  3  
  total: 1  6  9  4             total: 1  6  6  1 
      0: 1  .  .  .                 0: 1  .  .  .  
      1: .  .  .  .                 1: .  .  .  .  
      2: .  .  .  .                 2: .  .  .  .  
      3: .  6  9  4                 3: .  6  6  1  
\end{verbatim}
 \vskip 2mm
 In particular, the pruned resolution we obtain using Algorithm \ref{alg3} is minimal.
\end{example}

%
%
%
%

\section{Pruned resolution of monomial ideals and Betti splittings} \label{Sec2_2}

Francisco, H\`a and Van Tuyl  \cite{FTVT2}  introduced the notion of \textit{Betti splittings} building upon previous work of Eliahou and Kervaire \cite{EK}. In this section, we will apply this concept to present a version of the pruned resolution that is well-suited for inductive processes, as we will show in Section \ref{Sec6}. 

\begin{definition}
We say that $I=J+K$ is a {\it Betti splitting} for the monomial ideal $I$ if
the following formula for the $\bZ^n$-graded Betti numbers is satisfied:
 $$\beta_{i,\alpha}(I)= \beta_{i,\alpha}(J)+\beta_{i,\alpha}(K)+\beta_{i-1,\alpha}(J\cap
K)\,.$$
\end{definition}

Betti splitting techniques can be used recursively to compute Betti numbers of monomial ideals. However, such decomposition may be difficult to find or may not even exist.  For  squarefree monomials, a topological interpretation of Betti splittings is provided in \cite{BU}.
Francisco, H\`a and Van Tuyl  \cite[Corollary 2.4]{FTVT2} proved that we have a Betti splitting when $J$ and $K$ have a linear resolution. Bolognini \cite[Theorem 3.3]{Bolognini} extended this result  to the case where $J$ and $K$ are componentwise linear ideals.

\vskip 2mm

In \cite{FTVT2} they also studied the question whether a $x_i$-partition is a Betti splitting, which they refer as $x_i$-splitting. A $x_i$-partition $I=J+K$ is when $J$ is the ideal generated by all minimal generators of $I$ divisible by $x_i$, and $K$ is the ideal generated by all other generators of $I$.  When $I\subseteq R$ is the edge ideal of a simple graph then we always have a $x_i$-splitting.

\vskip 2mm

 We  give a mild generalization of \cite[Theorem 3.3]{Bolognini} for the case of $x_i$-partitions.

\vskip 2mm
\begin{lemma}\label{L} Let $J=(m_1,\dots,m_q)\subseteq R$ be a componentwise linear ideal minimally generated in degrees $d_1,\ldots,d_q$.
Then, $\beta_{i,i+j}(J)\neq 0$ if and only if $j=d_1,\ldots,d_q$.
\end{lemma}

\begin{proof}
By applying \cite[Proposition 2.2]{RS} if $j=d_1,\ldots,d_q$, then $\beta_{i,i+j}(J)\neq 0$. Conversely, suppose $\beta_{i,i+j}(J)\neq 0$ and consider $d_1\leq d_2\leq\ldots\leq d_q$. Set $J_1=J_{(d_1)}$. Therefore $J_1$ has $d_1$-linear resolution, and so $\beta_{i,i+j}(J_1)\neq 0$ if and only if $j=d_1$. Now, we consider $I=J/J_1$. Then $I$ is a componentwise linear module minimally generated in degrees $d_2,\ldots, d_m$. Therefore, by using the same proof of \cite[Proposition 2.2]{RS}
we have $Tor_i(\KK, J)_{i+j}=\oplus_{r=1}^t Tor_i(\KK,J_r)_{i+j}$, where $J_r$ has $d_r$-linear resolution for $r=1,\ldots,q$. In addition, we have $Tor_i(\KK,J_r)_{i+j}\neq 0$ if and only if $j=d_r$ for $r=1,\ldots,q$. Hence, if $\beta_{i,i+j}(J)\neq 0$, then $j=d_1,\ldots,d_q$, as required.
\end{proof}


\begin{theorem}\label{T}
Let $I=J +K$ be an $x_i$-partition of $I\subseteq R$ in which all elements of $J$ are divisible by $x_i$. If the minimal graded free resolution of $J$ is componentwise linear, then $I=J+K$ is a Betti splitting.
\end{theorem}

\begin{proof}
By \cite[Proposition 2.1]{FTVT2} we have to prove that the morphisms 
$$Tor_i(\KK,J\cap K)_\alpha\overset{{\psi}_{i,\alpha}}\longrightarrow 
 Tor_i(\KK,J)_\alpha\oplus Tor_i(\KK,K)_\alpha$$
are zero for all $\alpha=(\alpha_1,\ldots,\alpha_n)\in \ZZ^n$. Since all elements of $J$ and $J\cap K$ are divisible by $x_i$, it follows that all $\ZZ^n$-graded Betti numbers of $J$ and $J\cap K$ occur in degrees $\alpha\in \ZZ^n$  with $\alpha_i >0$ and none of the $\ZZ^n$-graded Betti numbers of $K$ do. In particular,  $\beta_{i,\alpha}(J\cap K)>0$ implies  $\beta_{i,\alpha}(K)=0$ for all $i$ and all $\ZZ^n$-degrees $\alpha$ so we only need to check the vanishing of the map $$Tor_i(\KK,J\cap K)_\alpha\longrightarrow Tor_i(\KK,J)_\alpha$$ 

Since $J\cap K\subseteq{\mathfrak{m} J}$, we have that the map $Tor_0(\KK,J\cap K)_\alpha\longrightarrow  Tor_0(\KK,J)_\alpha$ is zero.
The linearity defect of a componentwise linear ideal is zero \cite[Proposition 4.9]{YSq}, 
so we may use \cite[Lemma 2.8(ii)]{Nguyen} and conclude the vanishing of all morphisms ${\psi}_{i,\alpha}$.
\end{proof}



\vskip 2mm

When we have a Betti splitting it is natural to apply the pruning Algorithm \ref{alg1} to the ideals $J$, $K$ and $J\cap K$ separately and then define the pruned Betti numbers of $I$ as:
 $$\overline{\beta}_{i,\alpha}(I):= \overline{\beta}_{i,\alpha}(J)+\overline{\beta}_{i,\alpha}(K)+\overline{\beta}_{i-1,\alpha}(J\cap
K)\,.$$
We will show that this method gives a Morse resolution, in particular a cellular resolution, that starts with the Taylor complex $X_{\tt Taylor}$ of the ideal $I$.

 \vskip 2mm

 Let  $I=J+K$ be a  Betti splitting with $J=( m_1,\dots,m_s )$ and $K=
(m_{s+1},\dots,m_q)$. Notice that a (possibly non-minimal) set of generators of  $J\cap K$  is given by
 the monomials
 \begin{equation} \label{generators_intersection}
 \{ m_{1,s+1},\dots, m_{s,s+1},  m_{1,s+2}, \dots, m_{s,s+2},  \dots, m_{1,q}, \dots, m_{s,q}\}.
 \end{equation}
To compute the pruned
resolution of $J$ and $K$ we consider the subcomplexes:

\vskip 2mm

\begin{itemize}
 \item [$\cdot$] $X_J\subseteq X_{\tt Taylor}$  with faces labelled by
 $\sigma=(\sigma_1,\dots,\sigma_s,0,\dots,0) \in \{0,1\}^q$, and
 \item [$\cdot$] $X_K\subseteq X_{\tt Taylor}$  with faces labelled by
 $\sigma=(0,\dots,0,\sigma_{s+1},\dots,\sigma_q) \in \{0,1\}^q$.
\end{itemize}

\vskip 2mm

The set $X'$ obtained by removing the faces of
$X_J$ and $X_K$ from $X_{\tt Taylor}$
 is not a simplicial subcomplex of $X_{\tt Taylor}$ and
 it is not the Taylor simplicial complex associated to $J\cap K$. 
 It is proved in \cite{AMFG20} that, applying a partial pruning algorithm to $X_{J\cap K}$, we obtain  $X'$. 
 This shows that, even if the Taylor complex $X_{J\cap K}$ may be a large simplicial complex which is not a subcomplex of $X_{\tt Taylor}$, 
 we may apply the pruning Algorithm  \ref{alg1} within the edges  $E_{X'}\subseteq E_{X_{\tt Taylor}}$ in order to compute the pruned free resolution of $J\cap K$. Moreover, in the case where  the set of generators given in Equation \ref{generators_intersection} is not minimal, we may perform a partial pruning to eliminate the superfluous generators $m_{i,j}$ and the monomials $m_{i_1,\dots, i_s }$ containing $m_{i,j}$.

 \vskip 2mm
 
 \begin{algorithm}\label{alg2} 

\vskip 2mm

{\rm \noindent {\sc Input:} The set of edges $E_{X_{\tt Taylor}}$.


\vskip 2mm

Apply the pruning Algorithm \ref{alg1} to the subset of edges  $E_{X_{J}}$, $E_{X_{K}}$ and $E_{X'}$ separately.

\vskip 2mm

\noindent {\sc Return:} The set $\cA_B$ of edges that have been pruned.

}

\end{algorithm}

We interpret this algorithm as a partial pruning. In particular, $\cA_B$ is a homogeneous acyclic matching on $X_{\tt Taylor}$ and thus, 
 
 \begin{theorem}\label{C2}
Let $I \subseteq R=\KK[x_1, \ldots, x_n]$ be a monomial ideal and
$\cA_B \subseteq E_{X_{\tt Taylor}}$ be the set of pruned edges obtained using Algorithm \ref{alg2}.
Then, the $\bZ^n$-graded CW-complex $(X_{\cA_B},gr)$
supports a cellular free resolution $\mathbb{F}_{\bullet}^{(X_{\cA_B},gr)}$ of $R/I$.
\end{theorem}

%

The strength of this method is to apply it iteratively as long as we can find Betti splittings for the corresponding ideals $J,K$ and $J\cap K$. Notice that we will be always doing partial pruning on the original Taylor complex $X_{\tt Taylor}$ and thus we will get an acyclic matching and a cellular free resolution. 
\vskip 2mm

\begin{definition}
Let $I \subseteq R$ be a monomial ideal. We say that it {\it admits a minimal pruned resolution} if there exists an order of the generators of $I$ such that the cellular resolutions obtained using the iterative version of Theorem \ref{C2} is minimal.
\end{definition}

The following example illustrates this method (see \cite[Example 5.9]{AMFG20}).

\begin{example}
Consider the edge ideal of a $9$-cycle $$I=(x_1x_2,x_2x_3,x_3x_4,x_4x_5,x_5x_6,x_6x_7,x_7x_8,x_8x_9,x_9x_1) \subseteq \KK[x_1,\dots,x_9].$$
First we consider the $x_9$-splitting 
$$ J=(x_1x_2,x_2x_3,x_3x_4,x_4x_5,x_5x_6,x_6x_7,x_7x_8), \; K= (x_8x_9,x_9x_1).$$ 
The pruned resolutions of $J$ and $K$ obtained using Algorithm \ref{alg1} are minimal. 

\vskip 2mm 

A set of generators for $J\cap K$ is 
$ \{ m_{1,8},\dots, m_{7,8},  m_{1,9},\dots, m_{7,9}\}$ but we can remove $m_{1,8}=x_1x_2x_8x_9$ and $m_{7,9}=x_1x_7x_8x_9$ to make it minimal. Therefore we consider
$$J\cap K=
( \underbrace{x_2x_3x_{8}x_{9},\dots,x_{6}x_{7}x_{8}x_{9},
x_{7}x_{8}x_{9}}_{J'}, \underbrace{
x_1x_2x_{9},x_1x_3x_{4}x_{9},\dots,x_{1}x_{6}x_{7}x_{9}}_{K'}
)\,$$ 

Even though we removed some generators and, in consequence, the set $X' \subseteq X_{\tt Taylor}$ has been reduced, we do not get a minimal pruned resolution after applying Algorithm \ref{alg1} to $J\cap K$.
The iterative step here is to use a $x_1$-splitting $J\cap K= J'+K'$.
The pruned resolution of $J'$ and $K'$ are minimal. 

\vskip 2mm

 A
minimal set of generators for the intersection $J'\cap K'$ is
$$( x_1x_2x_3x_{8}x_9, \dots ,
x_1x_{6}x_{7}x_{8}x_9,x_{1}x_2x_{7}x_{8}x_9 ) = x_1x_8x_9 ( x_2x_3, \dots ,
x_{6}x_{7},x_2x_{7})$$ but
the pruning algorithm applied to this ideal is equivalent to the one
for the $6$-cycle $( x_2x_3,x_3x_{4}, \dots ,
x_{6}x_{7},x_2x_{7})$, which gives a minimal pruned resolution.
\end{example}

\section{Monomial ideals with minimal pruned resolution} \label{Sec6}

In this section we will use Betti splittings in an iterative way to obtain minimal pruned resolutions of several classes of monomial ideals.
The upshot of this method is to unify, under the same umbrella of pruned Morse resolutions, several classical known resolutions.

\subsection{Some classes of monomial ideals}
Let $R=\KK[x_1, \ldots, x_n]$ be a polynomial ring  over a field $\KK$ and let $I \subseteq R$ be a monomial ideal. We denote by $I_{(j)}$ the ideal generated by all homogeneous polynomial of degree $j$ belonging to $I$.  Let $G(I)=\{m_1,\dots, m_q\}$ be the set of minimal monomial generators of $I$ which we consider an ordered set. We will also consider the lexicographical order $>_{\rm lex}$ when comparing monomials.
For a  monomial  $m$, we set $max(m)=\max\{i \; | \; x_i$ divides $m\}$ and $min(m)=\min\{i \; | \; x_i$ divides $m\}$.

\begin{definition} \label{Def_ideals} Let $I\subseteq R$ be a monomial ideal. We say:
\begin{itemize}
 \item $I$ is a \textbf{lexsegment} ideal if for all monomials $m\in I$ and all monomials $m'$ with $\deg(m)=\deg(m')$ and $m' >_{\rm lex} m$ we have $m'\in I$.
 \item $I$ is a \textbf{strongly stable} ideal, if one has $x_i(m/{x_j})\in I$ for all $m\in G(I)$ and all $i<j$ such that $x_j$ divides $m$, (see \cite{H.1993}).
 \item $I$ is a \textbf{stable} ideal, if  $x_i(m/x_{max(m)})\in I$ for all $m\in G(I)$ and  all $i<max(m)$.
 \item  $I$ is a \textbf{linear quotients} ideal if the colon ideal $((m_1,\ldots,m_{i-1}):m_i)$ is a prime monomial ideal for every $2\leq i\leq q$. 
 \item $I$ is a \textbf{componentwise linear}  if $I_{(j)}$ has a linear resolution for all $j$. 

\end{itemize}
\end{definition}

For more details on the definitions and properties we refer to \cite{Bigatti}, \cite{HM-JA}, \cite{EK}, \cite{HerzogTakayama}, \cite{HHNagoya} (see also \cite{HHbook}). In particular we have the following hierarchy: 

\vskip 3mm

\begin{tikzcd}[ column sep=scriptsize, arrows=Rightarrow]
\text{Lexsegment} \arrow[r]  &\begin{tabular}{c}
Strongly \\
Stable
\end{tabular} \arrow[r] &
  \text{Stable} \arrow[r]   &
 \begin{tabular}{c}
Linear \\
Quotient
\end{tabular} \arrow[r]  & \begin{tabular}{c}
Componentwise \\
Linear
\end{tabular}  
\end{tikzcd}


\vskip 2mm

Let $p$ be either $0$ or a positive prime number. Given two nonnegative integers $a$ and $b$, we say that $a \preccurlyeq_p b$ if $\binom{b}{a} \neq 0$  mod $p$. Pardue \cite{Pardue} introduced the class of $p$-Borel monomial ideals.

\begin{definition}
Let $I\subseteq R$ be a monomial ideal where $\KK$ is an infinite field of any characteristic. Then

\begin{itemize}
  \item $I$ is a \textbf{$p$-Borel} fixed ideal, if for all $i<j$ and all $m\in G(I)$ one has $(x_i/{x_j})^a m\in I$ for all  $a \preccurlyeq_p b$ where $b$ is the largest integer such that $x_j^b$ divides $m$.

\end{itemize}
    
\end{definition}

Provided that the characteristic of the base field $\KK$ is zero, $0$-Borel ideals (or simply Borel ideals) coincide with the class of strongly stable ideals. This is no longer true if the characteristic is $p>0$. We point out that Betti numbers of stable ideals do not depend on the characteristic of the base field. However, Betti numbers of $p$-Borel ideals may depend on the characteristic \cite{CavigliaKummini}.

\vskip 2mm



\vskip 2mm

We will also include into the picture 
the following class of ideals introduced by Moradi and Khosh-Ahang \cite{MKH-Math.Scand}.

\begin{definition}
  Let $I\subseteq R$ be a monomial ideal. We say that $I$ is a \textbf{vertex splittable} ideal if it can be obtained using the following inductive process: 
\begin{itemize}
\item If either $I=(0)$, $I=R$ or $I=(m)$ where $m$ is a monomial.
  \item If there exists a variable $x_i$ and two vertex splittable ideals $J \subseteq R$ and $K \subseteq \KK[x_1,\dots,\hat{x_i},\ldots, x_n]$ such that $I=x_iJ+K$ with $K\subseteq J$ and $G(I)$ is the disjoint union of $G(x_iJ)$  and $G(K)$.
\end{itemize}
  
\end{definition}

\begin{example}
Examples of vertex splittable ideals include:

\begin{itemize}
\item {\it Cover ideals of chordal graphs} are  vertex splittable by \cite[Theorem 1]{W-Proc.AMS}. 

\vskip 2mm

\item {\it Cover ideals of sequentially Cohen-Macaulay bipartite graphs}  are vertex splittable by \cite[Theorem 2.10]{V-Arch.Math} and \cite[Theorem 2.3]{MKH-Math.Scand}.

\vskip 2mm

\item {\it Polymatroidal ideals}  are  monomial ideals generated in a single degree for which the set of the exponent vectors of the minimal  generators  is the set of bases of a discrete polymatroid (see \cite{HHPoly}). They are  vertex splittable by \cite[Lemma 2.1]{MNS-arXiv}. More generaly, 
componentwise polymatroidal ideals, meaning that $I_{(j)}$ is a polymatroidal ideal for all $j$, are vertex splittable  \cite[Proposition 2]{CrupiFicarra}.
\vskip 2mm

\item {\it {\bf t}-spread strongly stable}  ideals are vertex splittable by \cite[Proposition 1]{CrupiFicarra}.


\vskip 2mm

\item Given $\beta\in \{0,1\}^n$ we denote $\mathfrak{p}_{\beta}=(x_i \, | \, \beta_i=1)$ the corresponding prime monomial ideal. Francisco and Van Tuyl \cite{FVTNagoya} considered  ideals 
$$I=\mathfrak{p}_{\alpha_1}^{a_1} \cap \cdots \cap \mathfrak{p}_{\alpha_r}^{a_r},$$ satisfying $\mathfrak{p}_{\alpha_i}+\mathfrak{p}_{\alpha_j}=\mathfrak{m}$ for all $i\neq j$.  They appear in the study of fat points and thetraedral curves. Combining \cite[Theorem 3.1]{FVTNagoya} and \cite[Proposition 1]{CrupiFicarra} we have that they are vertex splittable.
\end{itemize}
\end{example}

\vskip 2mm

\subsubsection{Vertex splittable ideals}
Moradi and Khosh-Ahang \cite{MKH-Math.Scand} proved that vertex splittable ideals are  linear quotients. They also proved that they admit a Betti splitting and thus they behave well with respect to inductive processes.  We prove the following:

\begin{theorem} \label{thm1}
Let $I\subseteq R$ be a stable monomial ideal. Then $I$ is vertex splittable.
\end{theorem}

\begin{proof}
We proceed by induction on the number of generators of the stable ideal $I$.
Consider a $x_i$-partition  $I=x_iJ +K$, where $J\subseteq R$ and $K\subseteq \KK[x_1,\dots,\hat{x_i},\ldots, x_n]$. 
By using the definition of stability, we have that $J, K$ are stable as well so by induction they are vertex splittable. Moreover, given $m\in K \subseteq I$, we have  $x_i(m/x_{max(m)})\in I$. In particular $x_i m/x_{max(m)}\in x_iJ$ and thus $ m \in J$.
Therefore $I$ is vertex splittable because $K\subseteq J$ and we also have that $G(I)$ is the disjoint union of $G(x_iJ)$  and $G(K)$.
\end{proof}

The first main result in this section is: 

\begin{theorem} \label{thm2}
Let $I\subseteq R$ be a vertex splittable monomial ideal. Then $I$ admits a minimal pruned resolution.
\end{theorem}

\begin{proof}
We proceed by induction on the number of generators of the  vertex splittable ideal $I$. There exists a variable $x_i$ and two vertex splittable ideals $J \subseteq R$ and $K \subseteq \KK[x_1,\dots,\hat{x_i},\ldots, x_n]$ such that $I=x_iJ+K$ with $K\subseteq J$ and $G(I)$ is the disjoint union of $G(x_iJ)$ and $G(K)$. By induction, $J$ and $K$ admit a minimal pruned resolution. On the other hand 
$x_iJ\cap K=x_iK$  and thus it also admits a minimal pruned resolution by . Therefore we can conclude that $I$ admits a minimal pruned resolution.
\end{proof}

Since stable monomial ideals are vertex splittable, we can include the Eliahou-Kervaire resolution in the framework of pruned resolutions.

\begin{corollary} \label{CorEK}
Let $I\subseteq R$ be a stable monomial ideal. Then $I$ admits a minimal pruned resolution which is isomorphic to the Eliahou-Kervaire resolution. 
\end{corollary}




 

\subsubsection{Linear quotients ideals}

The second main result in this section deals with the case of linear quotients. We start with the following technical result:
\begin{lemma}\label{L1}
Let $I\subseteq R$ be an ideal admitting a minimal pruned resolution. Then, $mI$ also admits a minimal pruned resolution for each monomial $m\in R$. 
\end{lemma}

\begin{proof}
We only need to notice that an edge can be pruned when applying Algorithm \ref{alg1} to the ideal $I$ if and only if the corresponding edge in $mI$ can also be pruned.
\end{proof}

\vskip 2mm

\begin{theorem} \label{thm3}
Let $I\subseteq R$ be a linear quotients monomial ideal. Then $I$ admits a minimal pruned resolution.
\end{theorem}

\begin{proof}
We proceed by induction on the number of generators  of  $I=(m_1,\ldots,m_q)$ with the base case $q=2$ being clear. The decomposition $I=(m_1,\ldots, m_{q-1})+(m_q)$ gives a Betti splitting because 
 the ideals $(m_1,\ldots, m_{q-1})$ and $(m_q)$ are linear quotients, and thus componentwise linear, monomial ideals so we can apply \cite[Theorem 3.3]{Bolognini}. We have $(m_1,\ldots, m_{q-1})\cap(m_q)= m_q ((m_1,\ldots, m_{q-1}):m_q)= m_q \, \mathfrak{p}$, for some monomial prime ideal $\mathfrak{p}$, so it admits a minimal pruned resolution because we can apply Lemma \ref{L1} since $\mathfrak{p}$ admits a minimal pruned resolution. Therefore $I$ admits a minimal pruned resolution. 
\end{proof}

We point out that we can also include the Herzog-Takayama resolution in the framework of pruned resolutions.

\begin{corollary} \label{CorHT}
Let $I\subseteq R$ be a linear quotients monomial ideal. Then $I$ admits a minimal pruned resolution which is isomorphic to the Herzog-Takayama resolution. 
\end{corollary}

 Componentwise linear ideals may not admit a minimal pruned resolution. The following example was suggested to us by Bolognini (see \cite{BU}).

\begin{example}
Let $I\subseteq \KK[x_1,\dots,x_{12}]$ be the Alexander dual ideal of a triangulation of the lens space $L(3,1)$. 
It admits a Betti splitting in any characteristic (see \cite[Example 6.7]{BU}). The Betti tables when $\KK=\QQ$ and $\ZZ/(3)$ are respectively:

\vskip 2mm

\begin{verbatim}
         0  1   2  3  4 5                   0  1   2  3  4 5
  total: 1 54 108 66 12 1            total: 1 54 108 67 13 1
      0: 1  .   .  .  . .                0: 1  .   .  .  . .
      1: .  .   .  .  . .                1: .  .   .  .  . .
      2: .  .   .  .  . .                2: .  .   .  .  . .
      3: .  .   .  .  . .                3: .  .   .  .  . .
      4: .  .   .  .  . .                4: .  .   .  .  . .
      5: .  .   .  .  . .                5: .  .   .  .  . .
      6: .  .   .  .  . .                6: .  .   .  .  . .
      7: . 54 108 66 12 1                7: . 54 108 66 12 1
                                         8: .  .   .  .  1 .
                                         9: .  .   .  1  . .
\end{verbatim}

Notice that it has a linear resolution over $\QQ$ so, in particular, it is componentwise linear. However it does not admit a minimal pruned resolution.  
\end{example}

\vskip 2mm

\subsubsection{A class of $p$-Borel ideals} Given a monomial ideal $J=(m_1, \dots, m_q)$ we denote $J^{[p^r]}=({m_1}^{p^r}, \dots, {m_q}^{p^r})$. J\"ollenbeck and Welker \cite{JollenWelker}
considered the following class of $p$-Borel ideals: 
$$I= J_1^{[p^{r_1}]} \cdots J_k^{[p^{r_k}]}$$
where $0\leq r_1 < \cdots < r_k$, $J_i$ are Borel ideals such that the degree of any minimal monomial generator $m\in G(J_i)$ is  $ < p^{r_{i+1}-r_i}$. In the case that each $J_i$ is a principal Borel ideal (see \cite[Definition 6.15]{JollenWelker}) they constructed a minimal cellular resolution of $I$ using discrete Morse theory. Our next result ensures the existence of a minimal cellular resolution for a larger class of ideals. 

\begin{theorem} \label{thm_product}
Let $I=(m_1,\dots,m_q)\subseteq R$ and $J=(n_1,\dots,n_q)\subseteq R$ be monomial ideals such that the degree in the variables $x_i$ of each monomial $m_j$ is strictly smaller than the degree in $x_i$ of each monomial $n_k$. If both $I$ and $J$ admit a minimal pruned resolution, then $IJ$ also admits a minimal pruned resolution.    
\end{theorem}

\begin{proof}
In the construction of the Taylor graph of the ideal $$IJ=(m_1n_1,\dots,m_qn_1, \dots , m_1n_r,\dots,m_qn_r)$$ we notice that $${\rm lcm}(m_{i_1}n_{j_1}, \dots , m_{i_a}n_{j_a}) = {\rm lcm}(m_{i_1}n_{j_1}, \dots , m_{i_a}n_{j_a}, m_{i_{a+1}}n_{j_{a+1}})$$ if and only if both equalities   
$${\rm lcm}(m_{i_1}, \dots , m_{i_a}) = {\rm lcm}(m_{i_1}, \dots , m_{i_a}, m_{i_{a+1}})$$
$${\rm lcm}(n_{j_1}, \dots , n_{j_a}) = {\rm lcm}(n_{j_1}, \dots , n_{j_a}, n_{j_{a+1}})$$ are satisfied. This is due to the degree condition on the generators that we impose. 

\vskip 2mm
Then, to each cell of the CW-complex associated to the minimal pruned resolution of $J=(n_1,\dots,n_q)$ we may attach the Taylor graph of $I=(m_1,\dots,m_q)$ and perform the pruning Algorithm \ref{alg1}.  This will give a minimal pruned resolution of $IJ$.
\end{proof}



\begin{corollary} \label{pBorel}
Let $p>0$ be a prime integer and let $J_1,\dots, J_k\subseteq R$ be monomial ideals admitting a minimal pruned resolution. Then, the ideal   $$I= J_1^{[p^{r_1}]} \cdots J_k^{[p^{r_k}]}$$
where $0\leq r_1 < \cdots < r_k$, and the degree of any minimal monomial generator $m\in G(J_i)$ is  $ < p^{r_{i+1}-r_i}$, also admits a minimal pruned resolution.  
\end{corollary}

\begin{proof}
Notice that if $J$ admits a minimal pruned resolution, that $J^{[p^r]}$ also admits a minimal pruned resolution. Then the result follows applying Theorem \ref{thm_product} iteratively.    
\end{proof}

\vskip 2mm

The same kind of arguments in the proof of 
Theorem \ref{thm_product} also work in the following situation.

\begin{theorem} \label{thm_product_2}
Let $I=(m_1,\dots,m_q)\subseteq R$ and $J=(n_1,\dots,n_q)\subseteq R$ be monomial ideals in different sets of variables. If both $I$ and $J$ admit a minimal pruned resolution, then $IJ$ also admits a minimal pruned resolution.    
\end{theorem}

\vskip 2mm


\subsubsection{${\bf a}$-stable ideals} Let ${\bf a}=(a_1,\dots, a_n)\in \ZZ \cup \{\infty\}$. We say that $m=x_1^{\alpha_1}\cdots x_n^{\alpha_n}$ is an ${\bf a}$-monomial if $\alpha_i \leq a_i$ for $i=1,\dots, n$. We also say that a monomial ideal $I\subseteq R$ is an ${\bf a}$-ideal if it is minimally generated by ${\bf a}$-monomials. Gasharov, Hibi and Peeva \cite{GHPeeva} introduced the class of ${\bf a}$-stable ideals as a generalization of the usual stable ideals which corresponds to the case ${\bf a}=(\infty,\dots, \infty)$. The case ${\bf a}=(2,\dots, 2)$ corresponds to squarefree stable ideals considered in \cite{AHH}. 
A minimal free resolution of ${\bf a}$-stable ideals is given in \cite[Theorem 2.2]{GHPeeva}. The following result is an analogue in the setting of Morse resolutions.

\begin{proposition} \label{prop_GHP}
Let ${\bf a}=(a_1,\dots, a_n)\in \ZZ \cup \{\infty\}$ and $I\subseteq R$ be an ${\bf a}$-ideal. Let $J=I+(m_1,\dots, m_q) \subseteq R$ be an ideal where the minimal generators $m_i$ are not ${\bf a}$-monomials. If $J$ admits a minimal pruned resolution then $I$ also admits a minimal pruned resolution.  
\end{proposition}

\begin{proof}
 This is a direct consequence of Corollary \ref{prop_peeva}.   
\end{proof}

\subsection{Recursive minimal pruned resolutions}

Even in cases where the ideal is not vertex splittable, there are situations where we can still apply the same techniques to reduce the minimality of the pruned resolution of $I$ to that of a smaller ideal $J$.

\subsubsection{Artinian reductions}

Let $I\subseteq R$ be a monomial ideal that admits a decomposition
$I=J+(x_1^{a_1},\dots, x_n^{a_n})$, where $a_1,\dots , a_n$ are positive integer numbers. Without loss of generality we may assume that $a_i$ is larger than the degree of $x_i$ in any monomial in the set of minimal generators $G(J)$. This is the setup that we will consider throughout this subsection.

\vskip 2mm



The main result of the subsection is:

\begin{theorem} \label{thm_a}
Let $I\subseteq R$ be a monomial ideal that admits a decomposition
$$I=J+(x_1^{a_1},\dots, x_n^{a_n})$$ with $a_i \in \mathbb{Z}_{> 0}$.
Then $I$ admits a minimal pruned resolution if and only if $J$ admits a minimal pruned resolution.
\end{theorem}

\begin{proof}
Without any loss of generality we may assume that $a_i \in \mathbb{Z}_{> 0}$ is larger than the degree of $x_i$ in any generator of $J$. 
If $I$ admits a minimal pruned resolution, then $J$ also admits a minimal pruned resolution by applying Corollary \ref{cor_a} iteratively.   
\end{proof}

In the setting of Theorem \ref{thm_a} we will say that a minimal pruned resolution ideal plus powers also has a minimal pruned resolution. In particular a lexsegment plus powers or a stable plus powers ideals admit a minimal pruned resolution. These ideals have been extensively studied in the literature in relation to the lex-plus-powers conjecture.

\vskip 2mm

A consequence of the previous result is a generalization of the main result in \cite{FaridiArtinian} where they considered the case of $q=4$ generators. The following result cannot be improved to $q=6$ because of Example \ref{ex_q_6}.

\begin{corollary}
Let $I\subseteq R$ be a monomial ideal that admits a decomposition
$$I=J+(x_1^{a_1},\dots, x_n^{a_n})$$ with $J$ having  $q \leq 5$ generators. Then both $I$ and $J$ admit a minimal pruned resolution, and thus they have a minimal free resolution supported on a CW-complex.
\end{corollary}

\begin{proof}
If $J$ has $q \leq 5$ generators then it admits a minimal pruned resolution by Theorem \ref{thm_q_5}. 
$I$ also admits a minimal pruned resolution by
Theorem \ref{thm_a}.     
\end{proof}

 \subsubsection{Edge ideals of graphs}
Given a simple graph in a set of $n$ vertices we may associate a squarefree monomial ideal $I$ generated in degree two, where the generators of the form $x_ix_j$ correspond to the edges of the graph.

\vskip 2mm

For simplicity in the notation, consider a $x_n$-splitting $I=x_n J+K$ with $K \subseteq \KK[x_1,\dots, x_{n-1}]$ and $J=(x_{1}, \dots, x_{s})$ corresponding to the neighbors of the vertex $x_n$. We will assume that there are no edges between these neighbors. We have that $$x_n J\cap K= x_n((x_{1}, \dots, x_{s})\cap K)$$ admits a splitting $x_n J\cap K= M_1+(M_2+\cdots+M_s)$ with 
$$ M_k= (x_kx_n)\cap\left[ \left(\sum_{\substack{x_i,x_j \notin N(x_{n}) \\ x_i,x_j \notin N(x_{k})}}       
(x_{i}x_{j})\right) +  \left(\underset{x_{\ell} \in N_{G}(x_{k})}{\sum}(x_{\ell})\right)\right],$$ 
where the variable $x_t$ does not appear in the ideal $M_k$ for $1\leq t < k \leq s$. 
We may further decompose $x_n J\cap K$ by considering $M_1$,  $M_2+\cdots+M_s$ and $$M_1\cap (M_2+\cdots+M_s)= (M_1\cap M_2)+ \cdots + (M_1\cap M_s).$$
We continue this iterative process and we have a splitting at each step because of the assumption  that there are no edges between the neighbors $x_1,\cdots ,x_s$ (see \cite{AlvJof1} for details, where the Alexander dual version of this process was considered). The upshot of this iterated splitting process is that the Betti numbers of  $I=x_n J+K$ can be described
in terms of $J$, $K$ and all the 
intersections 
$$ M_{k_1}\cap \cdots \cap M_{k_m} = (x_{k_1}\cdots x_{k_m}x_n)\cap \left[ \left(\sum_{\substack{x_i,x_j \notin N(x_{n}) \\ x_i,x_j \notin N(x_{k_t}), \\t=1\dots,m}} (x_{i}x_{j})\right) +  \left(\sum_{\substack{ x_{\ell_t} \in N(x_{k_t}), \\t=1\dots,m}} (x_{\ell_1}\cdots x_{\ell_m})\right)\right].$$

We point out that they may exist vertices $x_{\ell_t}$ in the neighborhood of several $x_{k_t}$'s and thus the degree of the monomial $x_{\ell_1}\cdots x_{\ell_m}$ may be smaller than $m$.  

\vskip 2mm
Denote
$$I(H_{k_1,\dots, k_m})= \left(\sum_{\substack{x_i,x_j \notin N(x_{n}) \\ x_i,x_j \notin N(x_{k_t}), \\t=1\dots,m}} (x_{i}x_{j})\right) +  \left(\sum_{\substack{ x_{\ell_t} \in N(x_{k_t}), \\t=1\dots,m}} (x_{\ell_1}\cdots x_{\ell_m})\right).$$ This is the edge ideal of an hypergraph $H_{k_1,\dots, k_m}$ in the polynomial ring $R'=\KK[x_{s+1}, \dots, x_{n-1}]$ in $n-(s+1)$ variables.  Notice also that $M_{k_1}\cap \cdots \cap M_{k_m}$ admits a minimal pruned resolution whenever the ideal $I(H_{k_1,\dots, k_m})$ admits a minimal pruned resolution. 

 \begin{theorem} \label{thm_b}
 Let $I\subseteq R$ be the edge ideal of a graph. Consider the $x_n$-splitting $I=x_n J+K$ with $K \subseteq \KK[x_1,\ldots, x_{n-1}]$. Assume that there are no edges between the neighbors of $x_n$. Then $I$ admits a minimal pruned resolution if $K$ and all the ideals $I(H_{k_1,\dots, k_m})$, $1\leq k_1<\cdots < k_m\leq s$, admit a minimal pruned resolution.   
\end{theorem}

\begin{proof}
Consider the decomposition $I=x_n J+K$.  Clearly, $x_iJ$ admits a minimal pruned resolution. On the other hand $x_i J\cap K$ admits a minimal pruned resolution whenever
all the ideals $I(H_{k_1,\dots, k_m})$, $1\leq k_1<\cdots < k_m\leq s$, admit a minimal pruned resolution. Then the result follows.
\end{proof}

 Using this result iteratively, we may show that edge ideals of many classes of graphs  admit a minimal pruned resolution. In particular we recover \cite[Examples 5.8, 5.9]{AMFG20} and give a different approach to the free resolution of a forest given in \cite{BarileMacchia}.

 \begin{corollary} \label{C}
 Let $I\subseteq R$ be the edge ideal of a graph. Consider the $x_n$-splitting $I=x_n J+K$ with $K \subseteq \KK[x_1,\ldots, x_{n-1}]$. If the degree of the vertex $x_n$ is one, then $I$ admits a minimal pruned resolution if $K$ admits a minimal pruned resolution.   
\end{corollary}

\begin{proof}
 Assume $x_1$ is the only neighbor of $x_n$. Then  $x_n J\cap K= x_1x_n \cap K$ and thus it admits a minimal pruned resolution if $K$ does so.
\end{proof}

\begin{corollary} \label{Ctree}
Let $I\subseteq R$ be the edge ideal of a path, a tree or a forest. Then $I$ admits minimal pruned resolution.
\end{corollary} 

\begin{corollary} \label{Ccycle}
Let $I\subseteq R$ be the edge ideal of  a cycle. Then $I$ admits minimal pruned resolution.
\end{corollary}

\begin{proof}
 The ideal $K$ in the $x_n$-splitting $I=x_n J+K$ is the edge ideal of a path so it admits a minimal pruned resolution. We have $x_n J\cap K= x_n((x_{1}, x_{2})\cap K)= M_1+M_2$. Let $x_3, x_{n-1}$ be the neigbors of $x_1,x_2$ respectively. Then 
 $$M_1=(x_1x_n)\cap (x_3,x_4x_5, \ldots,x_{n-2}x_{n-1})$$
 $$M_2=(x_2x_n)\cap (x_{n-1},x_3x_4,x_4x_5, \ldots,x_{n-3}x_{n-2})$$
 $$M_1+M_2=(x_1x_2x_n)\cap (x_{n-1}x_3,x_3x_4, \ldots,x_{n-2}x_{n-1})$$ and thus $I(H_1), I(H_2)$ admit a minimal pruned resolution because are the edge ideals of paths and $I(H_{1,2})$ admits a minimal pruned resolution because is the edge ideal of a cycle with $n-3$ vertices so we proceed by induction.
\end{proof}

The following results do not need Theorem \ref{thm_b} since we have a nice description of the ideal $x_n J \cap K$.

\begin{corollary} \label{Cwheel}
Let $I\subseteq R$ be the edge ideal of   a wheel. Then $I$ admits minimal pruned resolution.
\end{corollary}
\begin{proof}
   Let $x_n$ be the vertex of degree $n-1$. The ideal $K$ in the $x_n$-splitting $I=x_n J+K$ is the edge ideal of a cycle so it admits a minimal pruned resolution. We have $$x_n J\cap K= x_n((x_{1}, \ldots, x_{n-1})\cap K)=x_nK$$ and thus it also admits a minimal pruned resolution.
\end{proof}

\begin{corollary} \label{Cbipartite}
Let $I\subseteq R$ be the edge ideal of  a  complete bipartite graph. Then $I$ admits minimal pruned resolution.
\end{corollary}
\begin{proof}
Let $\{x_1,\ldots, x_n,y_1,\ldots , y_m\}$ be the set of vertices of the graph. The ideal $K$ in the $x_n$-splitting $I=x_n J+K$ is a complete bipartite graph in a smaller set of vertices so it admits a minimal pruned resolution by induction. We have $$x_n J\cap K= x_n((y_{1}, \ldots, y_{m})\cap K)=x_nK$$ and thus it also admits a minimal pruned resolution.
\end{proof}

%
%
%
%
%
%

\section*{Acknowledgments}
We thank Davide Bolognini for some useful comments and providing us with some interesting examples. The second author would like to thank the Universitat Polit\`ecnica de Catalunya and Centre de Recerca Matem\`atica (CRM) for their hospitality and assistance.
\bibliographystyle{alpha}
\bibliography{References}

\end{document}